%% file: Hilbert_restrictions_revised_final.tex
\documentclass[11pt]{article} 

\usepackage[T1]{fontenc}
\usepackage[utf8]{inputenc} 

\usepackage{geometry}
\usepackage[parfill]{parskip} 

\input{packages.tex}
\input{commands.tex}

\input{theoremlayout.tex}

\titleformat{\section}
       {\centering \normalfont\scshape}
       {\thesection.}
       {0.5em}
       {}
       []

\titleformat{\subsection}[runin]
       {\normalfont\bfseries}
       {\thesubsection.}
       {0.5em}
       {}
       [.]
\titleformat{\subsubsection}[runin]
       {\normalfont\bfseries}
       {\thesubsubsection.}
       {0pt}
       {}
       []

\title{\large Kudla Millson lift of toric cycles and restriction of Hilbert modular forms.}
\author{\normalsize Romain Branchereau\footnote{ McGill, Burnside Hall
805 Sherbrooke Street West
Montreal, Quebec H3A 0B9 \\ Email: \url{branchereauromain.math@gmail.com}}}
\date{} 

\begin{document}
\maketitle
\abstract{Let $V$ be quadratic space of even dimension and of signature $(p,q)$ with $p \geq q>0$. We show that the Kudla-Millson lift of {\em toric} cycles - attached to algebraic tori - is a cusp form that is the diagonal restriction of a Hilbert modular form of parallel weight one. We deduce a formula relating the dimension of the span of such diagonal restrictions and the dimension of the span of toric and special cycles.
}
{
  \tableofcontents
}

\section{Introduction}
\subsection{Intersection numbers of geodesics on modular curves} \label{motivation}
Let $Y_0(p)=\Gamma_0(p) \backslash \HH$ be the open modular curve for some prime $p$. Let $\gamma_\8$ be the image in $Y_0(p)$ of the geodesic in $\HH$ from $0$ to $\8$. It defines a relative cycle 
\begin{align}
    \gamma_\8 \in \Zcal_1(Y_0(p),\partial Y_0(p),\Z).
\end{align} 
On the other hand, one can attach compact geodesics to a real quadratic field $\Q(\sqrt{D})$. Every ideal class $I$ in the narrow class group $\Ccal^+_D$ defines a closed and oriented geodesic $\gamma_{I}$ in the modular curve $Y_0(p)$. After taking linear combinations and twisting by an odd character $\psi \colon \Ccal^+_D \longrightarrow \C^\times$ we get a cycle
\begin{align}
    \gamma_{\psi}=\sum_{I \in \Ccal^+_D} \psi(I) \gamma_I \in \Zcal_1(Y_0(p)).
\end{align}
There is a natural action of the Hecke operators on these geodesics by acting on the endpoints in $\HH$, which gives an element $T_n\gamma_{\psi}\in \Zcal_1(Y_0(p))$. Moreover, we have a pairing in homology
\begin{align}
    \langle - , - \rangle \colon H_1(Y_0(p),\Z) \times H_1(Y_0(p),\partial Y_0(p),\Z) \longrightarrow \Z.
\end{align}
If two geodesics $\gamma_1$ and $\gamma_2$ in $Y_0(p)$ intersect transversely and in a compact set then $\langle \gamma_1,\gamma_2 \rangle = \sum_{z \in \gamma_1 \cap \gamma_2} \pm 1$ is the topological intersection number, where $\pm 1$ depends on the local orientation at the intersection point.  Darmon-Pozzi-Vonk prove the following in \cite[Theorem.~A]{DPV}.
\begin{thm*}[Darmon-Pozzi-Vonk]
If $p$ splits in $\Q(\sqrt{D})$, the modular form \begin{align}\label{dpveq}
    \Theta_{\gamma_{\8} \otimes \gamma_{\psi}}(\tau)=L_p(\psi,0)-2 \sum_{n =1}^\8 \langle \gamma_{\8},T_n\gamma_{\psi} \rangle q^n
\end{align} of weight $2$ and level $\Gamma_0(p)$ is the diagonal restriction of a $p$-stabilized Hilbert-Eisenstein series $E^{(p)}(\psi,\tau,\tau)$ for a subgroup $\SL_2(\Q(\sqrt{D}))$.
\end{thm*}

In \cite{rbrforum}, we showed how to recover (and generalize) the theorem of Darmon-Pozzi-Vonk by using the Kudla-Millson lift. The idea was to consider the embedding $F^\times \subset \SO(d,d)$ of a totally real field $F$ of degree $d$, which gives a non-compact cycle on the locally symmetric space of $\SO(d,d)$. The (regularized) integral of the Kudla-Millson theta lift over { this cycle} gives a generating series of intersection numbers generalizing the right handside of \eqref{dpveq}. On the other hand, the relation to a Hilbert-Eisenstein series follows from a Siegel-Weil formula and the seesaw
\[
\begin{tikzcd}
\SO(d,d) \arrow[dash, dr] & \SL_2(F) \\
F^\times \arrow[dash]{u} \arrow[dash, ru] & \SL_2(\Q). \arrow[dash]{u}
\end{tikzcd}
\]
In the case where $F$ is a real quadratic field, this { yields the theorem} of Darmon-Pozzi-Vonk above.
\begin{rmk} Another closely related seesaw appears in \cite{bcg}
    \[
\begin{tikzcd}
\GL_d(\Q) \arrow[dash, dr] & \GL_2(F) \\
F^\times \arrow[dash]{u} \arrow[dash, ru] & \GL_2(\Q), \arrow[dash]{u}
\end{tikzcd}
\]
where the right hand side is the diagonal restriction of (the same) Hilbert-Eisenstein series, and the left hand side is the evaluation of an Eisenstein class on the same cycle as above. The two seesaws are related by an embedding $\GL_d(\Q)$ in $\SO(d,d)$, and the Bergeron-Charollois-Garcia lift is closely related to the Kudla-Millson lift via the Mathai-Quillen form. We hope to explain this in more detail in future work. 
\end{rmk}
\subsection{Main result} In this paper, we replace the torus $F^\times \subset \SO(d,d)$ by a more general anisotropic maximal $\Q$-torus, of maximal $\R$-rank in an orthogonal group $\SO(p,q)$ of signature $(p,q)$ with $p \geq q >0$. We restrict ourselves to anisotropic tori, to avoid having to deal with the regularization of the integral as in \cite{rbrforum}. However, one can combine the result presented here with the result of {\em loc. cit.} to consider any maximal $\Q$-torus of maximal $\R$-rank, not necessarily anisotropic. 
\begin{rmk} The study of cycles attached to algebraic tori in orthogonal groups is also motivated by the recent work of Darmon-Gehrmann-Linowski. In \cite{dgl}, the authors define rigid meromorphic cocycles in arbitrary signature, generalizing the rigid meromorphic cocycles in signature $(2,2)$ studied in \cite{DPV}. They conjecture that evaluated on toric cycles, these cocycles should be algebraic. The seesaws presented here could be relevant.
\end{rmk}

 Let $(V,Q)$ be a non-degenerate rational quadratic space of dimension $2d$ and signature $(p,q)$, where we suppose that $p \geq q>0$. Let $\SO_V \subset \GL_{2d}$ be the orthogonal group of $V$ and let $\SO_V(\R)^+\simeq \SO(p,q)^+$ be the connected component containing the identity. Let $K_\8$ be a maximal compact subgroup $K_\8 \simeq \SO(p)\times \SO(q) \subset \SO_V(\R)^+$ and $\D^+ \simeq \SO_V(\R)^+/K_\8$ its associated symmetric space, which is a $pq$-dimensional Riemannian manifold.

{ Let $\varphi_\fin \in \Scal(V_{ \A_\fin})$ be a finite Schwartz function preserved by an open compact subgroup $K_\fin \subset \SO_V(\A_\fin)$.} We set $K=K_\8K_\fin$ and consider the adelic space 
\begin{align}
    Y \coloneqq \SO_V(\Q) \backslash \SO_V(\A)/K,
\end{align} which is a finite union of locally symmetric space $\Gamma \backslash \D^+$ for some congruence subgroup of $\Gamma \subset \SO_V(\Q)^+$. 
 
 In \cite{km2,km3}, Kudla and Millson construct an element
\begin{align}
\Theta^V_\varphi \in H^q(Y) \otimes M_{d} (\Gamma'_\Q)   
\end{align}
that realizes a lift from the homology to the space of modular forms of weight $d=\frac{p+q}{2}$ for $\Gamma'_\Q \subset \SL_2(\Q)$ depending on the choice of the Schwartz function. As for the modular curve, there is a homological pairing 
\begin{align}
    \langle - , - \rangle \colon H_q(Y,\Z) \times { H_{pq-q}(Y, \partial Y, \Z)} \longrightarrow \Z.
\end{align}
If the homology classes are represented by two smooth immersed submanifolds that intersect transversely and in a compact set, then the intersection number $\langle C_1,C_2 \rangle$ is the signed intersection number, as for geodesics. The main feature of the Kudla-Millson lift is that for a cycle $C$ in $H_q(Y)$ it has the Fourier expansion
\begin{align}
\Theta^V_\varphi(\tau,C) \coloneqq \int_C\Theta^V_\varphi(\tau)=c^0_\varphi(C)+ \sum_{n \in \Q_{>0}} \langle C,C_n(\varphi) \rangle q^n,
\end{align}
where the cycles $C_n(\varphi) \in \Zcal_{pq-q}(Y,\partial Y,\Z)$ are {\em special cycles} coming from embeddings $\SO(p-1,q) \hookrightarrow \SO(p,q)$. 

In this paper, we will evaluate the Kudla-Millson form on cycles attached to algebraic tori. Let $\Tbf \subset \SO_V$ be an anisotropic algebraic $\Q$-torus, maximal and of maximal $\R$-rank. Since we assumed that $p \geq q$, the orthogonal group $\SO_V$ is of real rank $q$. Hence, the (maximal) real rank of  $\Tbf$ is $q$. Let 
$$\chi = (\chi_\8,\chi_\fin) \colon \Tbf(\Q) \backslash \Tbf(\A) \longrightarrow \C^\times$$ be a character of finite order. The archimedean part $\chi_\8$ is a character on $\Tbf(\R)\simeq (\R^\times)^q \times (S^1)^{\frac{p-q}{2}}$. To avoid trivial cancellations of the integral, we need to make the assumption that the character is odd, {\em i.e.} it is the sign function at every real place, and trivial at every complex place. Let $K_T=K_{T,\8}K_{T,\fin} \subset \Tbf(\A)$ be such that $K_{T,\8} = K_\8 \cap \Tbf(\R)$ is maximal compact, and that $K_{T,\fin} \subset K_{\fin}$. Let $\Ccal_T$ be the finite group of double cosets
\begin{align}
    \Ccal_T\coloneqq\Tbf(\Q)^+\backslash \Tbf(\A_\fin)/K_{T,\fin}.
\end{align} We define
\begin{align}
    Y_T \coloneqq \Tbf(\Q)\backslash \Tbf(\A)/K_T=\bigsqcup_{I \in \Ccal_{T}} \Lambda \backslash \R_{>0}^q
\end{align}
where $\Lambda \coloneqq \Tbf(\Q)^+ \cap K_{T,\fin}$. The quotient is compact since $\Tbf$ is anisotropic. The embedding of $\Tbf(\A)$ in $\SO_V(\A)$ induces a map $Y_T \longrightarrow Y$, and the image of the connected component $\Lambda \backslash \R_{>0}^q$ associated to $I \in \Ccal_T$ defines a cycle $C_{T,I} \in \Zcal_q(Y)$.
Let us now suppose that $\chi_\fin$ is trivial on $K_{T,\fin}$. We can see the finite part $\chi_\fin$ as a function on $\chi_\fin \colon \Ccal_{T} \longrightarrow \C^\times$ and define the cycle
\begin{align}
    C_{\chi}=\sum_{I \in \Ccal_{T}} \chi_{\fin}(I)C_{T,I} \in \Zcal_q(Y).
\end{align}
By the work of Kudla-Millson, it is known that if $q$ is odd, then $\Theta^V_\varphi(\tau,C)$ is a cusp form for any cycle $C$; see \cite[Theorem.~2]{KMIHES}. We show that when we restrict to the cycles $C_\chi$, then the lift is always a cusp form.

\begin{thm} Let $V$ be an even dimensional quadratic space of signature $(p,q)$ with $p \geq q >0$ and let $C_\chi$ be the cycle attached to an anisotropic maximal $\Q$-torus of maximal real rank $q$. Then $\Theta^V_\varphi(\tau,C_\chi) \in S_d(\Gamma'_\Q)$ is a cusp form of weight $d$ for  $\Gamma'_\Q \subset \SL_2(\Q)$.
\end{thm}
\noindent  In the case of toric cycles $C_\chi$ attached to isotropic $\Q$-tori, as considered in \cite{rbrforum}, the lift is not a cusp form in general. In fact, for suitable lattice in $V$ the constant term of $\Theta_\varphi(\tau,C_\chi)$ is a partial Hecke $L$-function of $\chi$.
\subsubsection{} The crucial fact that we want to exploit in this paper, is that any even dimensional quadratic space can be obtained by restriction of scalars of an \'etale algebra with involution. More precisely, let $V$ be a quadratic space over $\Q$ as before. Let $\Tbf(\Q)$ a maximal $\Q$-torus in $\SO_V(\Q)$. Then, there exists an \'etale algebra $E$ of dimension $2d$ with involution $\epsilon$ and a $d$-dimensional subalgebra $F$ fixed by $\epsilon$, such that $(V,Q)\simeq \Res_{F/\Q}(E,Q_\alpha)$ where 
\begin{align}
    Q_\alpha(x,y)=\alpha (x \epsilon(y)+\epsilon(x)y)
\end{align}
for some $\alpha \in F^\times$. Moreover, we have $\Tbf(\Q)\simeq E^1$ where $E^1$ are the elements $x$ in $E$ of norm $x\epsilon(x)=1$.
The quadratic extension $E/F$ is a product of quadratic extension $E_i/F_i$ where $F_i$ is a field. If $E_i=F_i \times F_i$ is split, the involution permutes the two factors and $E_i^1 \simeq F_i^\times$. This is the case considered in \cite{rbrforum}. On the other hand, when $E_i/F_i$ is a field extension, the involution $\epsilon$ is the Galois involution $\Gal(E_i/F_i)$. This is the case we want to consider in this paper. The assumption that $\Tbf$ is $\Q$-anisotropic implies that none of the factors $E_i$ is split. On the other hand, the assumption that the real rank of $\Tbf$ is maximal implies that the fields $F_i$ are totally real.

Since the torus $\Tbf(\Q)$ is the restriction of scalars of $E^1$, we obtain the following.

\begin{thm}\label{diagrestthm}
    Let $\chi$ be a character on an $\Q$-anisotropic torus $\Tbf(\Q) \simeq E^1$ of maximal $\R$-rank. Suppose for simplicity that $E$ is a quadratic field extension of the totally real field $F$ of degree $d$. The generating series
    \begin{align} \label{equationmainthm}
     \Theta_\varphi(\tau,C_\chi)=\sum_{n \in \Q_{>0}} \langle C_{\chi} , C_n(\varphi) \rangle q^n=\Theta_\varphi(\tau,\cdots,\tau,\chi) \in S_d(\Gamma'_\Q)
    \end{align}
    is the diagonal restriction of a Hilbert modular form $\Theta_{\varphi}(\tau_1,\dots,\tau_d,\chi)$ of parallel weight one for some subgroup of $\SL_2(F)$.
\end{thm}
If $\Tbf(\Q) \simeq E^1$ is an anisotropic torus attached to a product of field extensions $E_i/F_i$, the seesaw argument is still valid.  The right-hand side becomes a sum of products of diagonally restricted Hilbert modular forms for $\SL_2(F_i)$. See Theorem \ref{mainthmgen} for the statement.

The theorem can be summarized by the following seesaw 
\[
\begin{tikzcd}
\SO_V(\Q) \arrow[dash, dr] & \SL_2(F) \\
E^1 \arrow[dash]{u} \arrow[dash, ru] & \SL_2(\Q). \arrow[dash]{u}
\end{tikzcd}
\]
\begin{rmk}\begin{enumerate}
\item The condition $p \geq q$ on the signature of $V_\R$ and on the maximality of the $\R$-rank of the torus are necessary to ensure that the cycle $C_{\chi}$ has dimension $q$ and can be paired with the Kudla-Millson form of degree $q$.
\item This construction does not generalize to the dual pair $\SO_V \times \Sp_r(\Q)$ for $r>1$, since the tori do not give cycles of appropriate dimension. As we will see, the dimension of the cycles $C_\chi$ satisfies $\dim_\R C_\chi \leq q$ with equality exactly when $\Tbf$ is maximally $\R$-split. On the other hand, when $r>1$ the Kudla-Millson forms are not of degree $q$, but of degree $rq$. Hence, the only way the degree of the forms can match the dimension of the cycles is when $r=1$.
\end{enumerate}
\end{rmk}

\subsection{Spans of diagonal restrictions and toric cycles} Let us now add the conditions that $V$ is of dimension $p+q>4$, where $q$ is odd. { Let $L \subset V$ be an even unimodular lattice and $\varphi_\fin=\id_{\Lhat}$ where $\Lhat=L \otimes \Zhat$}. Let $S$ be the set of all characters $\chi$ that are as above, odd and of finite order. Let 
\begin{align}
    S_T  \coloneqq \Span\left \{ \left . \Theta_\varphi(C_\chi) \right \vert \ \chi \in S \right \}\subset S_{d}(\SL_2(\Z))
\end{align} be the subspace of cusp forms spanned by the diagonal restrictions $\Theta_\varphi(C_\chi)$. Several authors (including \cite{lieis},\cite{bogoli},\cite{keapit} and \cite{yang} ) have considered spans of diagonal restrictions of Hilbert modular forms, mainly in the case of Hilbert-Eisenstein series.

On the other hand, is also natural to ask what part of the homology is spanned by cycles associated to the torus $\Tbf$. Let us define the subspace
\begin{align}
    H_T \coloneqq \Span \left \{ \left . C_\chi \ \right \vert  \chi \in S\right \} \subset H_q(Y,\C)
\end{align}
spanned by the toric cycles $C_\chi$. Let us also define
\begin{align}
    H_{\cycle} \coloneqq \Span \left \{ C_n(\varphi) \right \} \subset H_{pq-q}(Y,\C)
\end{align}
to be the span of the special cycles. The orthogonal complement is the set
\begin{align}
    H_{\cycle}^\perp \coloneqq \left \{ \left . C \in H_q(Y,\C) \right \vert \langle C,C_n(\varphi) \rangle=0 \ \textrm{for all} \ n \in \NN_{>0} \right \} \subset H_q(Y,\C).
\end{align}

With the previous conditions on the signature of $V$, the adjoint of the Kudla-Millson lift is injective by a result of Bruinier-Funke \cite{bfinj}. In section \ref{spans} we deduce the following.

\begin{cor}\label{corollary121} Suppose that $V$ is of dimension $p+q>4$, where $p \geq q>0$ and $q$ odd. Then
\begin{align}
    \dim \left ( S_{d}(\SL_2(\Z)) \right )-\dim (S_T) = \dim \left ( H_q(Y,\C) \right ) -\dim \left ( \Span \{ H_{\cycle}^\perp , H_T \} \right ).
\end{align}
\end{cor}
In particular, we would have $S_{d}(\SL_2(\Z))=S_T$ if and only if $H_q(Y,\C)=\Span \{ H_{\cycle}^\perp , H_T \}$.

{ \begin{rmk} Corollary \ref{corollary121} also holds for non-unimodular lattices, provided that the Kudla-Millson lift is injective. In particular, one could use stronger injectivity results of the Kudla-Millson lift due to Zuffetti, Metzler and Kiefer \cite{zuffetti, mz23, kz23}, generalizing the result of Bruinier-Funke in the unimodular case. 
\end{rmk}}

\subsection{Examples} Let us consider some examples in signature $(2,2)$. Let $V=\Mat_2(\Q)$ be the quadratic space with the quadratic form $\det$. For a suitable lattice $L$, the locally symmetric space attached to $\SO(2,2)$ is $Y_0(p) \times Y_0(p)$.

Let us first consider the \'etale algebra $E=\Q(\sqrt{D}) \times \Q(\sqrt{D})$ with involution being the Galois involution in both factors. After twisting by suitable characters, the toric cycle attached to $\Tbf(\Q) \simeq E^1$ is $\gamma_\psi \times \gamma_{\psi'}$, where $\gamma_{\psi}$ and $\gamma_{\psi'}$ are both attached to the same real quadratic field $K=\Q(\sqrt{D})$, where $p$ is split. This setting was considered by Darmon-Harris-Rotger-Venkatesh in \cite{dhrv}, where they show that the generating series
\begin{align}\label{eqdhrv}
    \Theta_{\gamma_{\psi} \otimes \gamma_{\psi'}}(\tau)= \sum_{n =1}^\8 \langle \gamma_{\psi},T_n\gamma_{\psi'} \rangle q^n
\end{align} is the diagonal restriction of a 'Hilbert modular form' for $\SL_2(\Q)\times \SL_2(\Q)$. In fact, they prove a more precise result as they express the generating series as the trace from level $pD^2$ to $p$ of a product of two weight one modular forms. The corresponding seesaw is 
\[
\begin{tikzcd}
\SO(2,2) \arrow[dash, dr] & \SL_2(\Q) \times \SL_2(\Q) \\
\Q(\sqrt{D})^1\times \Q(\sqrt{D})^1 \arrow[dash]{u} \arrow[dash, ru] & \SL_2(\Q), \arrow[dash]{u}
\end{tikzcd}
\]
where $\Q(\sqrt{D})^1$ are the elements of norm $1$ in $\Q(\sqrt{D})$.

Similarly, we can consider a biquadratic field $E=\Q(\sqrt{D_1},\sqrt{D_2})$ with the involution that sends $\sqrt{D_i}$ to $-\sqrt{D_i}$. The corresponding seesaw
\[
\begin{tikzcd}
\SO(2,2) \arrow[dash, dr] & \SL_2(\Q(\sqrt{D_1D_2})) \\
\Q(\sqrt{D_1},\sqrt{D_2})^1 \arrow[dash]{u} \arrow[dash, ru] & \SL_2(\Q), \arrow[dash]{u}
\end{tikzcd}
\]
where $\Q(\sqrt{D_1},\sqrt{D_2})^1$ are the elements of norm $1$. The image of the torus $\Q(\sqrt{D_1},\sqrt{D_2})^1$ is a product of geodesics $\gamma_1 \times \gamma_2$ where the geodesics are attached to $\Q(\sqrt{D_1})$ and $\Q(\sqrt{D_2})$ respectively. We discuss this example in Section \ref{lowrank}.

\paragraph{Acknowledgments.} We thank Henri Darmon for suggesting to look at these seesaws, Pierre Charollois, Luis Garcia, H\aa vard Damm-Johnsen and the referees for helpful comments.

\section{\'Etale algebras with involutions and algebraic tori} \label{etale}

In this section, we review the relations between algebraic tori in orthogonal groups and \'etale algebras with involutions, as explained in \cite{maxtor1}, \cite{maxtor2} and \cite{dgl}.

\subsubsection{} Let $E$ be a commutative $\Q$-algebra of even dimension $\dim_\Q(E)=2d$. It is said to be {\em \'etale over $\Q$} if $E \otimes \Qbar \simeq \Qbar^{2d}$. Equivalently, the \'etale algebra $E$ is a product of finitely many number fields $R$.  Let $\epsilon$ be a $\Q$-linear involution on $E$ and let $F \coloneqq E^\epsilon$ be the subalgebra fixed by $\epsilon$. We have the following three types 
\begin{itemize}
\item[$i)$] the involution $\epsilon$ preserves { each factor $R$ and $\restr{\epsilon}{R}\neq\idrm$},
\item[$ii)$] the involution $\epsilon$ does not preserve { one of the factors} $R$. In that case there is another factor $R'$ such that $R'=\epsilon(R) \simeq R$,
\item[$iii)$] the involution $\epsilon$ preserves the factor $R$ and $\restr{R}{\epsilon}=\idrm$.
\end{itemize}
Hence we can write $E$ as a sum $E=E_1 \times \cdots \times E_r$ of $\epsilon$-invariant subalgebras, where $E_i=R$ is a field in case $i)$ and $iii)$, and $E_i=R \times R'$ is a product of fields in case $ii)$. The fixed algebra $F$ is then the sum $F=F_1 \times \cdots \times F_r$ where $F_i=E_i^\epsilon$ is a number field.
From now on suppose furthermore that $F$ has degree $\frac{1}{2}[E\colon \Q]$ over $\Q$. In that case we only have case $i)$ or $ii)$ and we deduce the following.
\begin{prop}
There is an element $\delta \in F^\times$ such that $E$ is isomorphic to $F[\theta]/(\theta^2-\delta)$ and the involution $\epsilon$ sends $\theta$ to $-\theta$.
\end{prop}
\begin{proof}
    For the case $i)$ the field extension $E_i/F_i$ is of degree $2$. Hence we have $E_i \simeq F_i[\theta_i]/(\theta_i^2-\delta_i)$ for some $\delta_i$ in $F_i^\times \setminus (F_i^\times)^2$. In case $ii)$ note we can identify $E_i \simeq F_i \times F_i$. Then we can take $\delta_i=1$ so that we get an isomorphism
    \begin{align}
    F_i[\theta_i]/(\theta_i^2-1) \longrightarrow E_i, \qquad
    \theta_i \longmapsto (-1,1).
    \end{align}
\end{proof}

\subsection{\'Etale algebras as quadratic spaces} \label{subsecetale} For $\alpha$ in $F^\times$ we define a $\epsilon$-hermitian form on $E$ by
\begin{align}
E \times E \longrightarrow E, \qquad
(x,y) \longmapsto \alpha x\epsilon(y).
\end{align}
 It is preserved by the elements of norm $1$
\begin{align}
E^1 \coloneqq \left \{ \left . x \in E^\times \, \right  \vert \, x\epsilon(x)=1 \right \}.
\end{align} 
Suppose that $F$ is a field and $E/F$ is an \'etale algebra. As { mentioned} in the introduction, the case where $E=F \times F$ is the split algebra will be excluded, hence we will restrict ourselves to the case where $E=F(\theta)$ is a quadratic field extension of $F$. In order to work with orthogonal groups instead of unitary groups, we view $E$ as an $F$-vector space and let $Q_\alpha$ be the quadratic form obtained by composing the hermitian form with the trace
\begin{align} \label{quadform2}
Q_\alpha \colon E \times E \longrightarrow F, \qquad
(x,y) \longmapsto \alpha \Tr_{E/F}x\epsilon(y)=\alpha(x\epsilon(y) + \epsilon(x)y).
\end{align}
Let $\SO_E$ be the orthogonal group of this quadratic space. We view it as an algebraic group over $F$ whose $F$-points are
\begin{align}
    \SO_E(F) \coloneqq \left \{\left . g \in \GL_2(F) \right \vert Q_\alpha(gx,gy)=Q_\alpha(x,y) \right \}.
\end{align}
\begin{prop} \label{Uismaxtor}
    Let $\delta \in F^\times$ be such that $E=F(\theta)$ with $\theta^2=\delta$. The map
    \begin{align}
    E^1  \longrightarrow \SO_E(F), \quad 
    x+y\theta \longmapsto \begin{pmatrix}
        x & y\delta \\ y & x
    \end{pmatrix}
\end{align}
is a group isomorphism. Furthermore, the restriction of scalars $ \Res_{F/\Q} \SO_E$ is a $\Q$-torus of rank $d$.
\end{prop}
\begin{proof}
    The parameter $\alpha$ is irrelevant here so let us assume $\alpha=1$. With respect to the basis $F$-basis $\{1,\theta\}$ of $E$, the quadratic form $Q_\alpha$ has Gram matrix $\diag(1,-\delta)$. It is clear that $E^1 \subset \SO_E(F)$. On the other hand, for a matrix $g=\begin{pmatrix} a & b \\ c & d    
\end{pmatrix}
$ the condition $g^T\begin{pmatrix}
        1 & 0 \\ 0 & -\delta
\end{pmatrix}g=\begin{pmatrix}
        1 & 0 \\ 0 & -\delta
\end{pmatrix}$ implies $a^2-c^2\delta=1$, $b^2-d^2\delta=-\delta$ and $ab=cd\delta.$
By multiplying the last equation by $d$ on both sides and using that $ad-bc=1$ we find that $b=c\delta$. Using the last equation again, we find that $ab=db$. If $b \neq 0$, this implies $a=d$. If $b=0$ one can easily check that $g=\pm \id_2$, hence $a=d=\pm 1$.

Over $\Qbar$, we have an isomorphism of quadratic space
\begin{align}
    E \otimes \Qbar & \longrightarrow F \otimes \Qbar \oplus F \otimes \Qbar \nonumber \\
    x+\theta y & \longmapsto (x+\theta y, x-\theta y).
\end{align}
On the other hand, the orthogonal group $\SO_E((F \otimes \Qbar)^2)$ of $(F \otimes \Qbar)^2$ is isomorphic to $(F \otimes \Qbar)^\times$, where the isomorphism is given by the map
\begin{align}
    (F \otimes \Qbar)^\times & \longrightarrow \SO((F \otimes \Qbar)^2) \nonumber \\
    \lambda & \longrightarrow \begin{pmatrix}
        \lambda & 0 \\ 0 & \lambda^{-1}
    \end{pmatrix}.
\end{align}
Hence $(\Res_{F/\Q} \SO_E)(\Qbar) \simeq (F \otimes \Qbar)^\times \simeq \Qbar^{d}$, is a torus of rank $d$.
\end{proof}

\subsubsection{} By restriction of scalars from $F$ to $\Q$ we obtain a $2d$-dimensional quadratic space $(V,Q)=\Res_{F/\Q} (E,Q_\alpha)$ over $\Q$, where $V=E \simeq \Q^{2d}$ and the quadratic form is defined by 
\begin{align} \label{quadform}
Q(x,y) \coloneqq \Tr_{F/\Q} \circ Q_\alpha(x,y).
\end{align} In particular, we have $Q(x,x) = \Tr_{F/\Q}(\alpha \N_{E/F}(x))$. Moreover, by restriction of scalars we have an embedding
\begin{align} \label{emb1}
    \SO_E(F) & \hooklongrightarrow \SO_V(\Q).
\end{align}
By Proposition \ref{Uismaxtor}, the image of the embedding is an algebraic $\Q$-torus of rank $d$. Conversely, if $\Tbf$ is a maximal $\Q$-torus in $\SO_V$, then there is an \'etale algebra $E$ such that $E^1 \simeq \SO_{E}(F) \simeq \Tbf(\Q)$. We will recall the proof in Subsection \ref{maxtorss} (see Proposition \ref{converse}).
\begin{rmk}
    By abuse of notation we will occasionally also denote by $Q_\alpha$ the quadratic form over $\Q$, instead of $Q$. It should always be clear from the context if we consider the quadratic space over $F$ or over $\Q$.
\end{rmk}

\subsubsection{} Let us now consider a general \'etale algebra $E/F$, with $\epsilon$-invariant factors $E_i/F_i$. For any place $v$ of $\Q$, the involution on $E$ extends to an involution $\epsilon_v$ on $E_{\Q_v} \coloneqq E \otimes \Q_v$ with fixed subalgebra $F_{\Q_v} \coloneqq F \otimes \Q_v$. In the same way as before, the algebra with involution $E \otimes \Q_v$ is a $2d$-dimensional quadratic space over $\Q_v$. Let $w$ be a place of $F$. Let us write $F_w=(F_i)_w$ and $E_w=(E_i)_w$ for the completions at $w$. We set $E_v \coloneqq \prod_{v \mid w} E_w$ and $F_v\coloneqq \prod_{v \mid w} F_w$, which are naturally isomorphic to $E_{\Q_v}$ and $F_{\Q_v}$ respectively.

Let $w$ be a non-archimedean place of $F_i$, where $E_i=F_i(\theta_i)$. Then, either $E_w=F_w(\theta_i)$ is non-split and the involution sends $\theta_i$ to $-\theta_i$, or $E_w=F_w \times F_w$ is split and the involution permutes the two factors. At a non-split place $w$ we have
\begin{align}
    \SO_E(F_w) = \left \{ \left . x \in E_w^\times \right \vert \N_{E_w/F_w}(x)=1\right\}.
\end{align}
On the other hand, if $w$ is split we have $\SO_E(F_w) \simeq  F_w^\times.
$

Define the following sets of archimedean places of $F$:
\begin{align}    
    S_1 & \coloneqq \{ \ \textrm{real embeddings of $F$ that extend to real embeddings of $E$} \ \}, \nonumber \\
    S_2 & \coloneqq \{ \ \textrm{real embeddings of $F$ that extend to complex embeddings of $E$} \ \}, \\
    S_3 & \coloneqq \{ \ \textrm{(pairs) of complex embeddings of $F$} \ \}. \nonumber
\end{align}
We denote the cardinality of those sets by
\begin{align} \label{defni}
    n_k \coloneqq \vert S_k \vert.
\end{align} We have $d=n_1+n_2+2n_3$. For any archimedean place $\sigma$ of $F$, the completion $E_\sigma/F_\sigma$ is an \'etale algebra with involution and we have the following possibilities:
\begin{enumerate}
    \item $F_\sigma=\R$ and $E_\sigma=\R \times \R$ is the split algebra where the involution $\epsilon$ permutes the two factors. In that case, we have $E^1_\sigma \simeq \SO_E(F_\sigma)\simeq \R^\times$. This happens exactly when $\sigma \in S_1$.
    \item $F_\sigma=\R$ and $E_\sigma=\C$ is the algebra where the involution $\epsilon$ is complex conjugation. In that case, we have $E^1_\sigma \simeq \SO_E(F_\sigma)\simeq S^1$ where $S^1 \subset \C^\times$ is the unit circle. This happens when exactly when $\sigma \in S_2$.
    \item $F_\sigma=\C$ and $E_\sigma=\C \times \C$ is the split algebra where the involution $\epsilon$ permutes the two factors.In that case, we have $E^1_\sigma \simeq \SO_E(F_\sigma)\simeq \C^\times$. This happens for any complex place $\sigma \in S_3$ of $F$.
\end{enumerate}
By taking the product of all archimedean places we get
\begin{align} \label{dimen1}
    E_\8^1 = (\R^{\times})^{n_1} \times (S^1)^{n_2} \times (\C^\times)^{n_3} = (\R^{\times})^{n_1+n_3} \times (S^1)^{n_1+n_2}.
\end{align}

Let $E_\8 = \prod_\sigma E_\sigma$ and $F_\8 = \prod_{\sigma \in S} F_\sigma \simeq \C^{n_1} \times \R^{n_2+n_3}$. The embeddings of $E$ give us the natural isomophism $E_\R \simeq \prod_\sigma E_\sigma$ of algebras with involutions, that restricts to the natural algebra isomorphism $F_{\R} \simeq F_\8$. Hence, we have an isomorphism of quadratic spaces $(E_\R, Q_\alpha) \simeq \bigoplus (E_\sigma,q_{\sigma(\alpha)})$ where $Q_{\sigma(\alpha)}(x,y)=\Tr_{E_\sigma/F_\sigma}(\sigma(\alpha)x\epsilon(y))$.
For $\alpha \in F^\times$ let $r_\alpha$ (resp. $s_\alpha$) be the number of embeddings $\sigma$ in $S_2$ such that $\sigma(\alpha)$ is positive (resp. negative). Note that $n_2=r_\alpha+s_\alpha$.
\begin{prop} \label{sign1} The signature of $E_\R$ is $(n_1+2r_\alpha+2n_3,n_1+2s_\alpha+2n_3)$. 
\end{prop}
\begin{proof}
We have an isomorphism of quadratic spaces $(E_\R, Q_\alpha) \simeq \bigoplus (E_\sigma,Q_{\sigma(\alpha)})$. We only have to find the signature at each place $\sigma$. 
\begin{enumerate}
    \item If $\sigma \in S_1$, then $E_\sigma=\R \times \R$ and the involution permutes the two factors. Hence 
    \begin{align}
        Q_{\sigma(\alpha)}\left ( \begin{pmatrix}
            t_1 \\ t_1'
        \end{pmatrix},\begin{pmatrix}
            t_2 \\ t_2'
        \end{pmatrix}\right )=\sigma(\alpha)(t_1t_2'+t_1't_2)
    \end{align}
    and $E_\sigma$ has signature $(1,1)$.
    \item If $\sigma \in S_2$, then $E_\sigma=\C \simeq \R^2$ and the involution is the complex conjugation. Hence 
    \begin{align}
        Q_{\sigma(\alpha)}\left ( z_1,z_2\right )=\sigma(\alpha)\Tr_{\C/\R}(z_1\bar{z_2})=\sigma(\alpha)\left [\re(z_1)\re(z_2)+\im(z_1)\im(z_2) \right ],
    \end{align}
    and the signature of $E_\sigma$ is $(2,0)$ if $\sigma(\alpha)>0$ and $(0,2)$ if $\sigma(\alpha)<0$.
    \item If $\sigma \in S_3$, then $E_\sigma=\C\times \C$ and the involution permutes the two factors. Hence 
    \begin{align}
        Q_{\sigma(\alpha)}\left ( \begin{pmatrix}
            z_1 \\ w_1
        \end{pmatrix},\begin{pmatrix}
            z_2 \\ w_2
        \end{pmatrix}\right )=\sigma(\alpha)(z_1w_2+z_2w_1)
    \end{align}
    and $E_\sigma$ has signature $(2,2)$. 
\end{enumerate}
It follows that the signature of $E_\R$ is
\begin{align}
    n_1(1,1)+r_\alpha(2,0)+s_\alpha(0,2)+n_3(2,2)=(n_1+2r_\alpha+2n_3,n_1+2s_\alpha+2n_3)
\end{align}
\end{proof}
\subsection{\'Etale algebras of maximal tori} \label{maxtorss}

Let $(V,Q)$ be a non-degenerate quadratic $\Q$-space of dimension $2d$ with orthogonal group $\SO_V$. An algebraic $\Q$-group $\Tbf \subset \SO_V$ is a $\Q$-torus of rank $a$ if $\Tbf(\Qbar) \simeq \Qbar^a$. If $K$ is a field extension of $\Q$, we say that $\Tbf$ is $K$-split if $\Tbf(K)\simeq K^a$. A torus $\Tbf$ is maximal if { it is not contained in a torus of larger rank}. The rank of a maximal $K$-split torus $\Tbf' \subset \Tbf$ is called the $K$-rank of $\Tbf$. The orthogonal group $\SO_V$ is of $\R$-rank $q$ (since we assume that $p\geq q$). Hence, if $\Tbf$ is of maximal $\R$-rank, then the $\R$-rank of $\Tbf$ is $q$.
\subsubsection{} Consider the $\Q$-algebra $\End(V)$ of $\Q$-linear endomorphism of $V$. We view $\SO_V \subset \GL_{2d}$, so that we view $\End(V) \subset \Mat_{2d}(\Q)$. It is equipped with a natural involution $\epsilon_Q$ defined by 
\begin{align*}
    Q(xv,w)=Q(v,\epsilon_Q(x)w)
\end{align*}
for $x \in E$. Explicitely we have $\epsilon_Q(x)=A^{-1}_Qx^TA_Q$ where $A_Q$ denotes the Gram matrix of $Q$ and $x^T$ is the transpose of $x$. Let $E_T \subset \End(V)$ be the subalgebra consisting of all the $\Q$-endomorphisms $x \in \End(V)$ such that $xt=tx$ for any $t \in {\Tbf}(\Q)$. We will denote it simply by $E$ since there is no risk of confusion, and let $F$ be the subalgebra fixed by $\epsilon_Q$. We view $\Tbf(\Q) \subset \SO_V(\Q) \subset \End(V)$. The following proposition is proved in \cite[Proposition.~3.3]{maxtor1}. 

\begin{prop}\label{converse} Let $\Tbf$ be a maximal $\Q$-torus in $\SO_V$. The algebra $(E,\epsilon_Q)$ is an \'etale algebra with involution. We have $2\dim_\Q F=\dim_\Q E$. 
\end{prop} 
\begin{proof} 
Over $\Qbar$, the quadratic space $V$ is isomorphic to $\Qbar^{2d}$ with quadratic form $Q(v,v)=v_1v_{d+1}+v_2v_{d+2}+\cdots+v_dv_{2d}$ {\em i.e.} with Gram matrix $\begin{pmatrix}
    0 & \id_d \\ \id_d & 0
\end{pmatrix}$. With respect to this quadratic form, the torus can be diagonalized in $\SO_V(\Qbar)$ to
\begin{align}
    \Tbf(\Qbar) \simeq \left \{ \left . \diag(t_1, \dots,t_d,t_1^{-1},\dots,t_d^{-1}) \right \vert t_i \in \Qbar \right \} \simeq \Qbar^d
\end{align}
 Hence, the centralizer $E \otimes \Qbar \simeq \Qbar^{2d}$ consists of diagonal matrices in $\End(V_{\Qbar})$. Thus, $E$ is an \'etale algebra. We have $\epsilon_Q(g)=g^{-1}$ for any $g \in \SO_V(\Q)$.  In particular, we have $\epsilon_Q(t)=t^{-1}$ for any $t \in \Tbf(\Q)$. Hence, the involution $\epsilon_Q$ preserves $E$ since for any $x \in E$ we have
 \begin{align}
     \epsilon_Q(x)t=\epsilon_Q(xt^{-1})=\epsilon_Q(t^{-1}x)=t\epsilon_Q(x).
 \end{align}
Moreover, the involution permutes $a_i$ with $b_i$ in $\diag(a_1, \dots,a_d,b_1,\dots,b_d) \in E \otimes \Qbar$. It follows that
\begin{align}
    F \otimes \Qbar \simeq \left \{ \left . \diag(a_1, \dots,a_d,a_1,\dots,a_d) \right \vert a_i \in \Qbar \right \} \simeq \Qbar^d
\end{align}
and $2\dim_\Q F=\dim_\Q E=2d$.
\end{proof}
\subsubsection{} For $\alpha \in F^\times$ we have the quadratic form
\begin{align}
    Q_\alpha \colon E \times E \longrightarrow F, \qquad
    (x,y) \longmapsto Q_\alpha(x,y)=\alpha \Tr_{E/F}( x \epsilon_Q(y))
\end{align}
that we already defined in \eqref{quadform}. The restriction of scalars $\Res_{F/\Q} (E,Q_\alpha)$ is the quadratic $\Q$-vector space where the quadratic form is $\Tr_{F /\Q} \circ Q_\alpha$. Let us now prove that $\Tbf(\Q) \simeq E^1$, where $E=E_T$ is the \'etale algebra with involution defined above.
\begin{prop} \label{restr}
    We have $\Res_{F/\Q} (E,Q_\alpha) \simeq (V,Q) $ for some $\alpha$ in $F^\times$. Moreover, the torus $\Tbf(\Q)$ is isomorphic to $\SO_{E}(F)$.
\end{prop} 
\begin{proof}
    The algebra $E$ acts faithfully on $V$. Since they have the same dimension, $V$ is an $E$-module of rank $1$. Let $v_0$ in $V$ be a module generator, then we have an isomorphism of $\Q$ vector spaces $E \simeq V$ given by $e \mapsto ev_0$. We want to check that this is an isomorphism of quadratic spaces.

    The quadratic form $Q$ induces an isomorphism of $E$ with its dual
    \begin{align}
        f_Q \colon E & \longrightarrow E^\vee=\Hom_\Q(E,\Q) \nonumber \\
        x & \longmapsto f_Q(x)[y]=Q(xv_0,yv_0).
    \end{align}
    The map is $E$-linear in the sense that for every $e \in E$ we have $f_Q(ex)=\epsilon_Q(e)f_Q(x)$.
     On the other hand, the trace form also induces an $E$-linear isomorphism of $E$ with its dual
    \begin{align}
        f_{\Tr} \colon E & \longrightarrow E^\vee=\Hom_\Q(E,\Q) \nonumber \\
        x & \longmapsto f_{\Tr}(x)[y]=\Tr_{F /\Q}(x\epsilon_Q(y)).
    \end{align}
    It is also linear in the sense that $f_{\Tr}(ex)=\epsilon_Q(e)f_{\Tr}(x)$.
    Hence $f_{\Tr}^{-1} \circ f_Q$ is an $E$-linear automorphism of $E$, satisfying $(f_{\Tr}^{-1} \circ f_Q)(xe)=x(f_{\Tr}^{-1} \circ f_Q)(e)$. It follows that $f_{\Tr}^{-1} \circ f_Q(x)=\alpha x$ for some nonzero $\alpha$ in $E$, so that $f_Q(x)=f_{\Tr}(\alpha x)$. Hence, for any $x,y$ in $E$
    \begin{align}
     Q(xv_0,yv_0) = \Tr_{F /\Q}(\alpha x\epsilon_Q(y)).  
    \end{align}
    By the symmetry of $Q$, after setting $y=1$ we have
    \begin{align}
        0 & = Q(xv_0,v_0)-Q(v_0,xv_0) \nonumber \\
        & = \Tr_{F /\Q}(\alpha x)-\Tr_{F /\Q}(\alpha \epsilon_Q(x)) \nonumber \\
        & = \Tr_{F /\Q}(\alpha x)-\Tr_{F /\Q}(\epsilon_Q(\alpha) x) \nonumber \\
        & = \Tr_{F /\Q}((\alpha-\epsilon_Q(\alpha))x)
    \end{align}
    for any $x$ in $E$. Since the trace form is non-degenerate, it follows that $\alpha$ is in $F^\times$.

    Finally, let us prove that $\Tbf(\Q) \simeq \SO_E(F)$. We have shown that $E^1 \simeq \SO_E(F)$. On the one hand, for $t \in \Tbf(\Q)$ we have $Q(tv,w)=Q(v,t^{-1}w)$. On the other hand, by definition of the involution $\epsilon_Q$ we have $Q(tv,w)=Q(v,\epsilon_Q(t)w)$. Hence $\epsilon_Q(t)=t^{-1}$ and $t \in E^1$. This shows the inclusion $\Tbf(\Q) \subset E^1$, and the equality follows from the fact that $\Tbf(\Q)$ is a maximal $\Q$-torus contained in the (also maximal) torus $E^1$, see Proposition \ref{Uismaxtor}.
\end{proof}

\begin{prop} \label{rankcondi}
 If $\Tbf$ has maximal $\R$-rank, then $n_1=q$ and $F$ is totally real.
\end{prop}
\begin{proof} From Proposition \ref{restr} we have $\Tbf(\Q) \simeq \SO_{E}(F) \simeq E^1$. Let $n_k$ be defined as in \eqref{defni}: $n_1$ is the number of real embedding of $F$ that extend to real embeddings of $E$, $n_2$ is the number of real embeddings of $F$ that extend to a complex embedding of $E$ and $n_3$ is the number of (pairs) of complex embeddings of $F$. By \eqref{dimen1}, we have
    $\Tbf(\R) \simeq E_\8^1= (\R^{\times})^{n_1+n_3} \times (S^1)^{n_1+n_2}$, hence $a=n_1+n_3$ and $b=n_1+n_2$. On the other hand, by Proposition \ref{sign1}, the signature of $V_\R=E_\R$ is $(p,q)=(n_1+2r_\alpha+2n_3,n_1+2s_\alpha+2n_3)$, where $r_\alpha$ (respectively $s_\alpha$) is the number of places in $S_2$ for which $\sigma(\alpha)$ is positive (respectively negative). So if $\Tbf$ has maximal $\R$-rank, then $a=q$ and we must have $n_3=s_\alpha=0$.
\end{proof}

\section{Kudla-Millson theta correspondence}

\subsection{Weil representation} \label{subsecweil}
Let $(\Vcal,\Qcal)$ be a quadratic space of dimension $2m$ over a totally real field $\kbf$ of dimension $k$. Let $\SO_{\Vcal}$ be the orthogonal group of $\Vcal$. We will consider the following two cases:
\begin{itemize}
    \item[$-$] the field is $\kbf=\Q$ of degree $k=1$ and the quadratic space $(\Vcal,\Qcal)$ is an arbitrary quadratic space $(V,Q)$ like in the introduction, of dimension $2m=2d$. The group $\SO_{\Vcal}$ is the orthogonal group $\SO_V$.
    \item[$-$] the field $\kbf=F$ is an arbitrary totally real field of degree $k=d$. The quadratic space is $\Vcal=E$, where $E$ is an \'etale algebra viewed as quadratic space of dimension $2m=2$ over $F$, and equipped with the quadratic form $Q_\alpha$. The orthogonal group $\SO_{\Vcal}=\SO_E \simeq E^1$ is a torus.
\end{itemize}

\subsubsection{} Let $\Wcal=\Vcal \oplus \Vcal \simeq \Vcal \otimes \kbf^2$ be the $4m$-dimensional symplectic space over $\kbf$ with the symplectic form 
\begin{align}
    \Bcal \left (\begin{pmatrix}
        v_1 \\ v_2
    \end{pmatrix},\begin{pmatrix}
        w_1 \\ w_2
    \end{pmatrix} \right )=\Qcal(v_1,w_2)-\Qcal(w_1,v_2).
\end{align} Its symplectic group is $\Sp(\Wcal) \simeq \Sp_{4m}(\kbf)$. At a place $w$ of $\kbf$, let $\Vcal_w \coloneqq \Vcal \otimes k_w$ be the completion. There is a local projective unitary representation $\omega$ on the space $\Scal(\Vcal_w)$ of Schwartz-Bruhat functions, called the Weil representation. It is a projective representation in the sense $\omega(g_1g_2)=c(g_1,g_2)\omega(g_1)\omega(g_2)$ for some complex cocycle $c(g_1,g_2)$ satisfying $\vert c(g_1,g_2) \vert=1$. After passing to the adèles, we get a unitary representation
\begin{align}
    \omega \colon \Sp_{4m}(\kbf) \longrightarrow \Ucal(\Scal(\Vcal_\A)),
\end{align}
which is again only projective. However, for certain subgroups of $\Sp(\Wcal_\A)$ the cocycle is trivial and we obtain a true ({\em i.e.} non-projective) representation. Let us consider some special cases and give some concrete formulas for the Weil representation.

\subsubsection{}Consider the subgroup $\SL_2(\kbf) \subset \Sp_{4m}(\kbf)$, embedded as
\begin{align} \label{slemb}
    \begin{pmatrix}
        a & b \\ c & d
    \end{pmatrix} \hooklongrightarrow \begin{pmatrix}
        a & & & b & & \\  & \ddots & &  & \ddots & \\  & & a &  & & b \\ c & & & d & & \\  & \ddots & &  & \ddots & \\  & & c &  & & d
    \end{pmatrix}.
    \end{align} The subgroup $\SO_{\Vcal}(\kbf) \subset \Sp_{4m}(\kbf)$, embedded as
\begin{align} \label{orthemb}
    h \hooklongrightarrow \begin{pmatrix} h & 0 \\ 0 & h
    \end{pmatrix}
    \end{align}
commutes with $\SL_2(\kbf)$. Hence, we can embed the product $\SL_2(\kbf) \times \SO_{\Vcal}(\kbf)$ as a subgroup of $\Sp_{4m}(\kbf)$. After passing to the adèles, the projective representation of $\Sp_{4m}(\A_\kbf)$ restricts to a true representation
\begin{align} \label{adelicrep}
    \omega \colon \SL_2(\A_\kbf) \times \SO_\Vcal(\A_\kbf) \longrightarrow \Scal(\Vcal_\A).
\end{align} 

Let us describe this action more precisely, on a Schwartz function $\varphi=\varphi_\8 \otimes \varphi_\fin \in \Scal(\Vcal_\R) \otimes \Scal(\Vcal_{\A_\fin})$. For $h \in \SO_{\Vcal}(\A_\kbf)$, we have  \begin{align}
    (\omega(1,h)\varphi)(v) & =\varphi(h^{-1}v).
\end{align}
Suppose we can decompose the Schwartz function as $\varphi= \otimes_w \varphi_w$ where $\varphi_{w}$ is in $\Scal(\Vcal_w)$ and the product is over places of $\kbf$. Let us write down the local Weil representation of $\SL_2(\kbf_w)$ on $\varphi_w$. 
If $g=\begin{pmatrix}
        a & 0 \\ 0 & a^{-1}
    \end{pmatrix}$ for some $a \in \kbf_w^\times$, then
\begin{align}
    (\omega(g,1)\varphi)(v) & = \vert a \vert_w^{m} \varphi_w(a v).
\end{align}
If $g=\begin{pmatrix}
        1 & b \\ 0 & 1
    \end{pmatrix}$ for some $b \in \kbf_w$, then
\begin{align}
    (\omega(g,1)\varphi)(v) & = {\chi_w \left (\frac{b\Qcal(v,v)}{2} \right )}\varphi(v).
\end{align}
Finally, if $S=\begin{pmatrix}
        0 & -1 \\ 1 & 0
    \end{pmatrix}$, then
\begin{align}
    (\omega(S,1)\varphi)(v) & =  \int_{\Vcal_{w}} \varphi_w(u)\chi_w(\Qcal(v,u))du_w.
\end{align}
\begin{rmk} The character $\chi_w$ is defined as follows. We fix the additive character $e_v$ on $\Q_v$ defined by
\begin{align}
e_v(x) \coloneqq  \begin{cases} e^{2i\pi x} & \textrm{if} \; v=\8 \\
e^{-2i\pi \{x\}_p} & \textrm{if} \; v=p,
\end{cases}
\end{align}
where $\{x\}_p$ is the fractional part of $x$ in $\Q_p$. We extend it to a character $\chi_w$ on $\kbf_w$ by setting $\chi_w \coloneqq e_v \circ \Tr_{\kbf_w \mid \Q_v}$. The Haar measure $du_w$ is the unique Haar measure on $\kbf_w$ which is self dual with respect to $\chi_{w}$. This is the Haar measure normalized such that the Fourier inversion holds.
\end{rmk}

\subsection{(Co)homology of adelic spaces} \label{subsecadelic} For every real place $\sigma$ of $\kbf$, let $\Vcal_\sigma \coloneqq \Vcal \otimes_\sigma \R$ and $\Vcal_\R \coloneqq \bigoplus \Vcal_\sigma$. It is a real quadratic space, and let $(p_\sigma,q_\sigma)$ be its signature. Let $\D_{\sigma}^{+}$ be the associated connected symmetric space, that can be described as one of the two connected { components} of the Grassmanian $\D_\sigma$ of $q_\sigma$-dimensional oriented subspaces $z \subset \Vcal_\sigma$ that are negative, {\em i.e.} $\restr{\Qcal}{z}<0$. We set $\D^+=\prod \D_{\sigma}^{+}$, where the product ranges over the archimedean places of $\kbf$ for which $p_\sigma q_\sigma$ is nonzero\footnote{In the case where $p_\sigma=0$ or $q_\sigma=0$, the manifold $\D_{\sigma}^{+}$ is just a point.}. The dimension of $\D^+_\sigma$ is $p_\sigma q_\sigma$, so that the dimension of $\D^+$ is $\sum_\sigma p_\sigma q_\sigma$. At every place $\sigma$ let $\kbf_\sigma$ be the completion of $\kbf$, so that $\kbf_\8 \simeq \prod_\sigma \kbf_\sigma$. Since we assumed that $\kbf$ only has real places, we have $\kbf_\sigma \simeq \R$. At the level of the orthogonal group, we have $\SO_{\Vcal}(\kbf_\8) \simeq { \prod_{\sigma} \SO_{\Vcal}(\kbf_\sigma)}$, where $\SO_{\Vcal}(\kbf_\sigma) \simeq \SO(p_\sigma,q_\sigma)$. The connected component $\SO_{\Vcal}(\kbf_\8)^+ \simeq \prod_\sigma \SO(p_\sigma,q_\sigma)^+$ of the identity acts transitively on $\D^+$. Let $K_\8$ be a maximal connected compact subgroup of $\SO_{\Vcal}(\kbf_\8)^+$, that is isomorphic to  $ \prod_{\sigma}\SO(p_\sigma) \times \SO(q_\sigma)$. Hence 
\begin{align}
    \D^+ \simeq \SO_{\Vcal}(\kbf_\8)^+/K_\8 \simeq \prod_\sigma \SO(p_\sigma,q_\sigma)^+/ \SO(p_\sigma) \times \SO(q_\sigma).
\end{align}
\begin{rmk} Note that when $\kbf=\Q$ and $\Vcal=V$, then for the unique real place we have $\D_{\sigma}=\D$, where 
\begin{align}
    \D=\left \{ \left (z,o) \right \vert z \subset V_\R, \ \dim(z)=q, \ \restr{Q}{z}<0, \ o \textrm{ an orientation of} \ z \right \} \subset \operatorname{Gr}_q(V_\R)
\end{align}
is the space { mentioned} in the introduction. It has two connected components (corresponding to the two choices of orientations) and $\D^+$ is one of them.   
\end{rmk}
\subsubsection{}Let $\varphi_\fin \in \Scal(\Vcal_{\A_{\kbf,\fin}})$ be a finite Schwartz function, let $K_\fin \subset \SO_{\Vcal}(\A_{\kbf,\fin})$ be an open compact preserving $\varphi_\fin$ in the sense that $\omega(k,1)\varphi_\fin=\varphi_\fin$. Let $K \coloneqq K_\8 K_\fin$ and consider the adelic space
\begin{align}
   Y=\SO_{\Vcal}(\kbf)\backslash \SO_{\Vcal}(\A_\kbf)/K \simeq \SO_{\Vcal}(\kbf)\backslash \D \times \SO_\Vcal(\A_{\kbf,\fin})/K_\fin.
\end{align}
Let $\Ccal$ be the finite group of double cosets
\begin{align}
    \Ccal \coloneqq \SO_\Vcal(\kbf)^+\backslash \SO_\Vcal(\A_{\kbf,\fin})/K_\fin.
\end{align}
Then we have
\begin{align}
 \SO_\Vcal(\A_{\kbf,\fin}) = \bigsqcup_{I \in \Ccal} \SO_\Vcal(\kbf)^+ h_I K_\fin.    
\end{align}
where $h_I \in \SO_\Vcal(\A_{\kbf,\fin})$ are representants. The adelic space is a finite union of locally symmetric spaces
\begin{align}
Y\coloneqq \bigsqcup_{I \in \Ccal} \Gamma_{I} \backslash \D^+   
\end{align}
where $\Gamma_{I} \coloneqq h_I K_\fin h_I^{-1} \cap \SO_\Vcal(\kbf)^+$.

\subsubsection{} The space of differential $r$-forms on $Y$ is defined by
\begin{align}
    \Omega^r(Y)\coloneqq \bigoplus_{I \in \Ccal} \Omega^r(\Gamma_I \backslash \D^+).
\end{align} Let $C^\8(\SO_\Vcal(\A_{\kbf,\fin}))$ be the space of locally constant functions on $\SO_\Vcal(\A_{\kbf,\fin})$. The map
\begin{align} \label{isoform2}
    \left [ \Omega^r(\D^+) \otimes_\Q C^\8(\SO_\Vcal(\A_{\kbf,\fin})) \right ]^{\SO_{\Vcal}(\kbf) \times K_\fin}  \longrightarrow \Omega^r(Y) \end{align}
sending $\eta \otimes f$ to $\sum_I f(h_I)\eta$ is an isomorphism, where $C^\8(\SO_\Vcal(\A_{\kbf,\fin}))$ is the space of locally constant functions. We define the homology and cohomology of $Y$ by
\begin{align}
 H^r(Y)\coloneqq \bigoplus_{I \in \Ccal} H^r(\Gamma_I \backslash \D^+), \quad   H_r(Y)\coloneqq \bigoplus_{I \in \Ccal} H_r(\Gamma_I \backslash \D^+),
\end{align}
and similarly for the compactly supported cohomology. The integral of a closed form $\eta=\sum_{I \in \Ccal} \eta_I$ over a cycle $C=\sum_{I \in \Ccal} C_{I}$ in $Y$ is then defined by
\begin{align}
    \int_C\eta=\sum_{I \in \Ccal} \int_{C_I}\eta_I.
\end{align} 
This pairing induces the Poincaré duality $H_r(Y) \simeq H^r(Y)^\vee \simeq H^{\dim(Y)-r}_c(Y)$.

For top degree forms, when $r=\dim(Y)$, the choice of an orientation $\varrho$ gives an isomorphism
\begin{align} \label{isoform1}
    C^\8\left (\SO_\Vcal(\kbf_\8)^+ \right )^{K_\8} \longrightarrow \Omega^{\dim(Y)}(\D^+)
\end{align}
that sends a smooth $K_\8$-invariant function $f$ to a top degree form $f\varrho$. Combining the two isomorphisms \eqref{isoform2} and \eqref{isoform1}, we get the following isomorphism for top degree forms
\begin{align}\label{topdegreeiso}
 C^\8\left (\SO_\Vcal(\kbf)\backslash \SO_\Vcal(\A_\kbf) \right)^{K}   \longrightarrow \Omega^{\dim(Y)}(Y), 
\end{align}
sending a function $f$ to the form $\sum_{I \in \Ccal} f(\cdot,h_I)\varrho$. When $\eta$ is a top degree form, we can consider the integral over $Y$. Suppose that the form $\eta \in \Omega^{\dim(Y)}(Y)$ { corresponds} to a function $f$ in the isomorphism \eqref{topdegreeiso}. Then
\begin{align}
    \int_{Y}\eta = \frac{c}{\vol(K)}\int_{\SO_\Vcal(\kbf)\backslash \SO_\Vcal(\A_\kbf)}f(g)dg
\end{align}
where $\vol(K)$ is the volume of $K$ with respect to a Haar mesure $dg$ on $\SO_{\Vcal}(\A_\kbf)$, and $c>0$ is a constant dependant on the choice of the Haar measure but independant of $K$. We suppose that the Haar measure is chosen such that $c=1$.
\subsection{Kudla-Millson theta lift} \label{subseckm}
Let ${ (p,q) =\sum_\sigma (p_\sigma,q_\sigma)}$ be the signature of $\Vcal_\R = \bigoplus \Vcal_\sigma$. In \cite{km2,km3} Kudla and Millson define a form \begin{align}
    \varphi_{\KM} \in \Omega^{q}(\D^+,\Scal(\Vcal_\R))^{\SO_\Vcal(\kbf_\8)^+} \simeq \left [\Omega^{q}(\D^+) \otimes \Scal(\Vcal_\R) \right ]^{\SO_\Vcal(\kbf_\8)^+}
\end{align} valued in the Schwartz space $\Scal(\Vcal_\R)$. More precisely, at every place $\sigma$ there is a form $\varphi_{\KM}^{\sigma} \in \Omega^{q_\sigma}(\D_\sigma^+,\Scal(\Vcal_\sigma))^{\SO_\Vcal(\kbf_\sigma)^+}$ such that $\varphi_{\KM}(v)=\bigwedge_{\sigma=1}^k \varphi_{\KM}^{\sigma}(\sigma(v))$. These forms (hence also $\varphi_{\KM}$) are closed {\em i.e.} $d\varphi_{\KM}^{\sigma}(v)=0$ for any $v$ in $\Vcal$, and $\SO_\Vcal(\kbf_\sigma)^+$-invariant in the sense that 
\begin{align}
    h^\ast \varphi_{\KM}^{\sigma}(v)=\varphi_{\KM}^{\sigma}(h^{-1}v)
\end{align} for any $h$ in $\SO_\Vcal(\kbf_\sigma)^+$ and $v$ in $\Vcal$. This extends to the invariance property
\begin{align} \label{invariance}
    h^\ast \varphi_{\KM}(v)=\varphi_{\KM}(h^{-1}v)
\end{align} for any $h$ in $\SO_\Vcal(\kbf_\8)^+$.
In particular, the Kudla-Millson form is $\SO_{\Vcal}(\kbf)^+$-invariant. Furthermore, it satifies
\begin{align}
\omega(k_\theta)\varphi_{\KM}=\prod_\sigma e^{i \theta_\sigma m} \varphi_{\KM} 
\end{align} where { $m=\dim(\Vcal)$} and
\begin{align}
    k_\theta= \left (\begin{pmatrix}
    \cos(\theta_\sigma) & \sin(\theta_\sigma) \\ -\sin(\theta_\sigma) & \cos(\theta_\sigma)
\end{pmatrix} \right )_\sigma \in \SO(2)^k.
\end{align}
One of the main features of the Kudla-Millson form is its Thom form property. We will come back to this in  Subsection \ref{sec:genseries}.
\subsubsection{} Let us now define the Kudla-Millson theta series. Let $\varphi_\fin \in \Scal(\Vcal_{\A_{\kbf,\fin}})$ be the same Schwartz function as above, and let $K'_\fin \subset \SL_2(\A_{\kbf,\fin})$ be an open compact satisfying $\omega(1,k')\varphi_\fin(v)=\varphi_\fin(v)$ for every $k' \in K'_\fin$. Hence, the Schwartz function $\varphi_\fin$ is $K_\fin \times K'_\fin$ invariant by the Weil representation. We define
\begin{align}
    \varphi \coloneqq \varphi_{\KM} \otimes \varphi_\fin \in \Omega^q(\D^+) \otimes \Scal(V_{\A_\kbf})
\end{align}
For $g \in \SL_2(\A_\kbf)$ and $h_\fin \in \SO_\Vcal(\A_{\kbf,\fin})$, the Kudla-Millson theta series is
\begin{align} \label{thetakernel}
    \Theta_\varphi(g,h_\fin) \coloneqq  \sum_{v \in \Vcal} (\omega(g,h_\fin)\varphi )(v) \in C^\8 \left ( \SL_2(\kbf)\backslash \SL_2(\A_\kbf)\right ) \otimes \Omega^q(Y).
\end{align}
When $g$ is fixed, we can view 
\begin{align}
    \Theta_\varphi(g) \coloneqq \Theta_\varphi(g, \cdot) \in \Omega^q(Y)
\end{align}
as a differential form on $Y$. Let $\Gamma'_\kbf \coloneqq \SL_2(\kbf) \cap K'_\fin$. For a point $(\tau_1, \dots, \tau_k)\in \HH^k$ let $g_{\8}=(g_{\tau_1}, \dots, g_{\tau_k}) \in \SL_2(\R)^k$ where
\begin{align}
    g_{\tau_\sigma} = \begin{pmatrix}
        \sqrt{y_\sigma} & \nicefrac{x_\sigma}{\sqrt{y_\sigma}} \\ 0 & \nicefrac{1}{\sqrt{y_\sigma}}
    \end{pmatrix} \in \SL_2(\R),
\end{align}
and $\tau_\sigma=x_\sigma+iy_\sigma$. Let
\begin{align} \label{kmlift2}
    \Theta_\varphi(\tau_1, \dots, \tau_k) & \coloneqq (y_1 \cdots y_k)^{-\frac{m}{2}} \Theta_\varphi(g_\8,\cdot) \in \Omega^q(Y),
\end{align}
where we recall that $m=\dim(\Vcal)$. By the work of Kudla and Millson - building on Weil's construction of automorphic forms - the form $\Theta_\varphi$ transforms like a Hilbert modular form of parallel weight $m$ and level $\Gamma'_\kbf \subset \SL_2(\kbf)$, in the variables $\tau_1, \dots, \tau_k$. They also show that the form is holomorphic in cohomology, in the sense that for every $\sigma$ we have
\begin{align}
 \frac{\partial}{\partial \bar{\tau}_\sigma}\Theta_\varphi(\tau_1, \dots, \tau_k)=d \eta_\sigma   
\end{align}
for some $\eta_\sigma \in \Omega^{q-1}(Y)$. Furthermore, the form $\Theta_\varphi$ is closed, since $\varphi_{\KM}$ is. Hence, we can now view $\Theta_\varphi$ as an element 
\begin{align}
    \Theta_\varphi \in H^q(Y) \otimes M_{(m, \cdots ,m)}(\Gamma'_\kbf )
\end{align}
where $M_{(m, \cdots ,m)}(\Gamma'_\kbf )$ is the space of Hilbert modular forms of parallel weight $m$ and level $\Gamma'_\kbf \subset \SL_2(\kbf)$.
In particular, if $C \in \Zcal_q(Y)$ is a cycle, then 
\begin{align} \label{kmlift1}
    \Theta_\varphi(\tau_1, \dots, \tau_k,C) & = \int_{C} \Theta_\varphi(\tau_1, \dots, \tau_k) \nonumber \\
   & = \sum_{I \in \Ccal}\int_{C_I} \Theta_\varphi(\tau_1, \dots, \tau_k,h_I) \in M_{(m, \cdots ,m)}(\Gamma'_\kbf ).
\end{align}
Equivalently, if $\eta_C \in \Omega_c^{\dim(Y)-q}(Y)$ is a compactly supported form of complementary degree that is a Poincaré dual to $C$, then 
\begin{align} 
    \Theta_\varphi(\tau_1, \dots, \tau_k,C) = \int_{Y} \Theta_\varphi(\tau_1, \dots, \tau_k) \wedge \eta_C \in M_{(m, \cdots ,m)}(\Gamma'_\kbf ).
\end{align}
\subsection{Restriction of scalars and seesaw} 
 We can now consider the theta lift described in the previous paragraph in two cases.
 
 \begin{rmk}[Remark on the notation] In section \ref{subsecweil} we introduced  the notations $\D$, $Y$ and $\Ccal$ attached to the quadratic space $\Vcal$ where $\Vcal$ is either the quadratic space $V$ over $\Q$ or the quadratic space $E$ over $F$. In the later case, we will replace the notation by $\D_T$, $Y_T$ , $\Ccal_T$ and $K_T$, as used in the introduction. { We will use the notations $\Theta_\varphi^V$ and $\Theta_\varphi^T$ if we need to emphasize the underlying dual pair.}
\end{rmk}

\subsubsection{} First, let us consider the case where $\kbf=\Q$ and $\Vcal=V$. The symplectic space is $\Wcal \simeq \Q^{4d}$ (recall that the degree of $F$ is $d$). We have the Weil representation
\begin{align}
 \omega_V \colon \SL_2(\A) \times \SO_V({\A}) \longrightarrow \Scal(V_\A).   
\end{align}
Let $\varphi_\fin \in \Scal(V_{\A_\fin})$ a finite Schwartz function as previously. Let $\varphi=\varphi_{\KM}^V \otimes \varphi_\fin$, where $\varphi_{\KM}^V \in \Omega^{q}(\D^+)$ is the Kudla-Millson form on $\D^+$. In this setting, the Kudla-Millson lift is
\begin{align} \label{lift3}
    \Theta_{\varphi} \colon H_q(Y) \longrightarrow M_{d}(\Gamma'_\Q),
\end{align}
for some $\Gamma'_\Q \subset \SL_2(\Q)$ depending on $\varphi_\fin$.
\subsubsection{}   
Let us now consider the Kudla-Millson lift for the quadratic space $(\Vcal,\Qcal)=(E,Q_\alpha)$. Let us consider subsections \ref{subsecweil} and \ref{subsecadelic} in the setting of this quadratic space of dimension $m=2$ over the totally real field $\kbf=F$ of degree $k=d$. Recall that $\Tbf(\Q) \simeq E^1 \simeq \SO_E(F)$. 

The symplectic space is $\Wcal \simeq F^4$, and we have a representation
\begin{align}
 \omega_E \colon \SL_2(\A_F) \times \SO_E(\A_F) \longrightarrow \Scal(\A_E).   
\end{align}
Let $\Tbf(\R)^+ \simeq \SO_E(F_\R)^+$ be the connected components of its real points. Since $\Tbf$ is of maximal real rank, we have $E_\8^{1,+} \simeq \Tbf(\R)^+\simeq (\R_{>0})^q \times (S^1)^\frac{p-q}{2}$ where $S^1$ is the unit circle. The maximal compact subgroup in $\Tbf(\R)^+$ is $K_{T,\8} \simeq (S^1)^\frac{p-q}{2}$, hence
\begin{align}
    \D_T^+=\Tbf(\R)^+ / K_{T,\8} \simeq (\R_{>0})^q.
\end{align}

We can view the Schwartz function $\varphi_\fin \in \Scal(V_{\A_\fin})$ used for the lift \eqref{lift3} as a finite Schwartz function in $\varphi_\fin \in \Scal(\A_{E,\fin})$ and set $\varphi=\varphi_{\KM}^E \otimes \varphi_\fin$, where $\varphi_{\KM}^E \in \Omega^{q}(\D_T^+) \simeq \Omega^{q}(\R_{>0}^{q})$ is the Kudla-Millson form on $\D_T^+$.

Let $K_{T,\fin} \times K_{T,\fin}' \subset \SO_E(\A_{F,\fin}) \times \SL_2(\A_{F,\fin})$ be the open compacts stabilizing $\varphi$ under the Weil representation, and let $Y_T$ the locally symmetric space. 
Let $\Ccal_T$ be the finite group of double cosets
\begin{align}
    \Ccal_T \coloneqq \Tbf(\Q)^+\backslash \Tbf(\A_\fin)/K_{T,\fin}=E^{1,+} \backslash \A_{E,\fin}^1/K_{T,\fin}
\end{align}
where $E^{1,+}= E^1 \cap (E^1_\8)^+ \simeq \Tbf(\Q) \cap \Tbf(\R)^+$ is the intersection with the connected component of the identity. The space
\begin{align}
    Y_T=\Tbf(\Q)\backslash \Tbf(\A)/K_{T} \simeq E^1 \backslash E^1_\A/K_T \simeq E^1 \backslash \left ( (\R_{>0})^{q} \times \A_{E,\fin}^1 \right )/K_{T,\fin}
\end{align}
is a disjoint union of connected components
\begin{align} \label{cup1}
    Y_T=\bigsqcup_{I \in \Ccal_T} \Lambda \backslash \R_{>0}^{q}.
\end{align}
where $\Lambda \coloneqq \Tbf(\Q)^+ \cap K_{T,\fin}$. By \cite[Theorem.~4.11, p.~208]{pr}, the quotient $Y_{T}$ is compact since $\Tbf(\Q) \simeq E^1$ is $\Q$-anisotropic. Then $\dim(Y_T)={q}$ and the Kudla-Millson theta lift is
\begin{align} \label{lift1}
    \Theta^T_{\varphi}\colon H_{q}(Y_T) \longrightarrow M_{(1,\dots,1)}( \Gamma'_F),
\end{align}
where $\Gamma'_F \subset \SL_2(F)$.

We have $H_{q}(Y_T,\C)\simeq \C[\Ccal_T]$. Let $Y_{T,I}$ be the connected component corresponding to $I$ in \eqref{cup1} and $[Y_{T,I}] \in H_{q}(Y_{T,I})$ a fundamental class, {i.e.} a generator of $H_{q}(Y_{T,I},\Z) \simeq \Z$. Let $\chi_\fin$ be the finite part of the Hecke character and suppose that $\chi_\fin$ is trivial on $K_{T,\fin}$. Hence, we can view $\chi_\fin$ as a function on $\Ccal_T$ and define
\begin{align}
    [Y_\chi] \coloneqq \sum_{I \in \Ccal_T} \chi_\fin(I)Y_{T,I} \in H_{q}(Y_T,\C).
\end{align}
We define the image of this class to be
\begin{align} \label{integraldef}
    \Theta^T_{\varphi}(\tau_1, \cdots, \tau_{d},\chi) \coloneqq \int_{Y_\chi} \Theta^T_{\varphi}(\tau_1, \cdots, \tau_{d}) \in M_{(1,\dots,1)}( \Gamma'_F).
\end{align}
 Note that $\D_T^+ \simeq \SO_E(\R)^+ \simeq (\R_{>0})^{q}$. By \eqref{topdegreeiso}, the orientation $dt^\times \coloneqq dt_1^\times \wedge \cdots \wedge dt_{q}^\times$ of $\D_T^+$, where $dt_\sigma^\times=\frac{dt_\sigma}{t_\sigma}$, identifies
\begin{align}
    C^\8 \left ( \Tbf(\Q) \backslash \Tbf(\A) \right )^{K_T} \simeq \Omega^{q}(Y_T).
\end{align} Hence, we can write
\begin{align} \label{lift2}
    \Theta^T_{\varphi}(\tau_1, \cdots, \tau_{d})= \thetil^T_\varphi(\tau_1, \cdots, \tau_d,t)dt^\times \in \Omega^{q}(Y_T)
\end{align}
for some smooth function $\thetil_{\varphi}(\tau_1, \cdots, \tau_d, \cdot ) \in C^\8(\Tbf(\Q) \backslash \Tbf(\A))^{K_T}$. We can also write \eqref{integraldef} as
\begin{align}
    \Theta^T_{\varphi}(\tau_1, \cdots, \tau_{d},\chi) = \frac{1}{\vol(K_{T})}\int_{\Tbf(\Q) \backslash \Tbf(\A)} \thetil^T_\varphi(\tau_1, \cdots, \tau_{d}, t ) \chi(t) dt^\times \in M_{(1,\dots,1)}( \Gamma'_F),
\end{align}
where $\vol(K_{T})$ is the volume with respect to the Haar measure $dt^\times$ on $\Tbf(\A) \simeq E^1_\A$.

\subsubsection{} The seesaw identity is a way to relate the lifts \eqref{lift1} and \eqref{lift3}. By restriction of scalars from $F$ to $\Q$, we can embedd $\Sp_{4}(F)$ in $\Sp_{4d}(\Q)$, as well as the subgroups forming the dual pair $\SL_2(\A_F) \times \SO_E(\A_F)$. The image of these two subgroups satisfy $\SL_2(\Q) \subset \SL_2(F)$ and $\SO_E(F) \subset \SO_V(\Q)$, where the latter inclusion is restriction of scalars. For the former embedding of $\SL_2(F)$ in $\Sp_{4d}(\Q)$, it is given by
\begin{align}
    \SL_2(F) & \hooklongrightarrow \Sp_{4d}(\Q) \nonumber \\
    \begin{pmatrix}
        a & b \\ c & d
    \end{pmatrix} & \longmapsto \begin{pmatrix}
        r(a) & r(b) \\
        r(c) & r(d)
    \end{pmatrix}
\end{align}
where $r \colon F^\times \longrightarrow \GL_d(\Q)$ is the regular representation. These two dual pairs form the seesaw
\[
\begin{tikzcd}[sep=large]
\SO_V(\Q) \arrow[dash, dr] & \SL_2(F) \\
\SO_E(F) \arrow[dash]{u} \arrow[dash, ru] & \SL_2(\Q). \arrow[dash]{u}
\end{tikzcd}
\]
We also had the pairs of open compact subgroups $K_{\fin} \times K_{\fin}' \subset \SO_V(\A_\fin) \times \SL_2(\A_\fin)$  and $K_{T,\fin} \times K_{T,\fin}' \subset \SO_E(\A_{F,\fin}) \times \SL_2(\A_{F,\fin})$ stabilizing $\varphi_\fin$. Suppose they are compatible with the embeddings in the sense that $K_{T,\fin} \subset K_{\fin} \cap \SO_E(\A_{F,\fin}) $ and $K_{\fin}' \subset K_{T,\fin}' \cap \SL_2(\A_\fin)$. After conjugating if necessary, we can assume that the maximal compacts $K_\8$ and $K_{T,\8}$ are chosen such that $K_\8 \cap \SO_V(\R)^+=K_{T,\8}$. Then the embedding of $\SO_E$ in $\SO_V$ induces an immersion
\begin{align}
    \phi \colon Y_T \longrightarrow Y.
\end{align}
This immersion induces a pullback in cohomology
\begin{align}
    \phi^\ast \colon H^{q}(Y,\C) \longrightarrow H^{q}(Y_T,\C),
\end{align}
and a pushforward in homology
\begin{align}
    \phi_\ast \colon H_{q}(Y_T,\C) \longrightarrow H_{q}(Y,\C).
\end{align}
On the other hand, the pullback by the diagonal inclusion of $\HH$ in $\HH^d=\HH \times \cdots \times \HH$ induces the diagonal restriction
\begin{align}
    \iota^\ast \colon M_{(1,\dots,1)}( \Gamma'_F) \longrightarrow M_{d}(\Gamma'_F \cap \SL_2(\Q)).
\end{align}
{  Note that $\Gamma'_\Q \subset \Gamma'_F \cap \SL_2(\Q)$ by the assumption on the compact groups $K_\fin'$ and $K_{T,\fin}'$.}
The seesaw identity relies on two observations. Firstly, the two Weil representations $\omega_V$ and $\omega_E$ agree on their smallest common subgroups of both dual pairs, namely $(\SO_E(F),\SL_2(\Q))$. Secondly, the Kudla-Millson form has the functorial property that $\restr{\varphi_{\KM}^V}{\D_T^+}=\varphi_{\KM}^E$, see \cite[Equation.~(5.10)]{rbrkmmq}. Hence, the two kernels
\begin{align}
    \Theta^V_{\varphi} \in C^\8 \left ( \HH\right ) \otimes \Omega^q(Y)
\end{align}
and 
\begin{align}
    \Theta^T_{\varphi} \in C^\8 \left ( \HH^d \right ) \otimes \Omega^q(Y_T)
\end{align}
for $E$ and $V$ respectively, agree when restricted to 
\begin{align}
    C^\8 \left (  \HH \right ) \otimes \Omega^{q}(Y_T).
\end{align}
We deduce the seesaw identity
\begin{align} \label{seesaw}
  \iota^\ast \Theta^T_{\varphi} = \phi^\ast \Theta^V_{\varphi}.
\end{align}
\subsubsection{} For a class $I$ in the class group $\Ccal_T$ of $\Tbf$ we define the toric cycle
\begin{align}
    C_{T,I} = \phi_\ast Y_{T,I} \in H_{q}(Y,\Z)
\end{align}
to be the pushforward of the fundamental class. We also define
\begin{align} \label{cycledef}
    C_\chi \coloneqq \phi_\ast Y_\chi=  \sum_{I \in \Ccal_T} \chi_\fin(I) C_{T,I} \in H_{q}(Y,\Z).
\end{align} We then have
\begin{align}
    \Theta^T_{\varphi}(\tau, \cdots, \tau,\chi) = \int_{Y_\chi} \Theta^T_{\varphi}(\tau, \cdots, \tau) = \int_{Y_\chi} \phi^\ast \Theta^V_\varphi(\tau) = \int_{\phi_\ast Y_\chi} \Theta^V_\varphi(\tau),
\end{align}
which gives the seesaw identity
\begin{align} \label{seesaw1}
    \Theta^T_{\varphi}(\tau, \cdots, \tau,\chi) = \int_{C_\chi} \Theta^V_{\varphi}(\tau).
\end{align}

\subsubsection{}For now we have only considered the case where $E$ is an extension of a field $F$, but we want to consider more general \'etale algebras of the form $E=E_1 \times \cdots \times E_r$, together with an involution fixing $F=F_1 \times \cdots \times F_r$. Here each of the $F_i$ is a totally real field of degree $d_i$, that is the fixed field of $E_i$. Suppose that $(V,Q)$ is the restriction of scalars of $(E,Q_\alpha)$ from $F$ to $\Q$, as in \ref{subsecetale}. Then $V$ is a vector space of dimension $2d$ where $d=d_1+\cdots +d_r$. Furthermore, we can split the quadratic space as $V=V_1 \oplus \cdots \oplus V_r$, where $V_i$ is the restriction of scalars of $E_i$ from the field $F_i$ to $\Q$. We can restrict the embedding $\SO_V(\Q) \hookrightarrow \Sp_{4d}(\Q)$ given in \ref{orthemb} to an embedding $\SO_{V_1}(\Q) \times \cdots \times \SO_{V_r}(\Q) \hookrightarrow \Sp_{4d}(\Q)$. Its commutator is $\SL_2(\Q)^r$ and we get a seesaw
\[
\begin{tikzcd}
\SO_V(\Q) \arrow[dash, dr] & \SL_2(\Q)^r \\
\SO_{V_1}(\Q) \times \cdots \times \SO_{V_r}(\Q) \arrow[dash]{u} \arrow[dash, ru] & \SL_2(\Q). \arrow[dash]{u}
\end{tikzcd}
\]
Combining with previous seesaw, we then get
\[
\begin{tikzcd}
\SO_V(\Q) \arrow[dash, dr] & \SL_2(F_1) \times \cdots \times \SL_2(F_r)  \\
\SO_{E_1}(F_1) \times \cdots \times \SO_{E_r}(F_1) \arrow[dash]{u} \arrow[dash, ru] & \SL_2(\Q) . \arrow[dash]{u}
\end{tikzcd}
\]

Let $\varphi_\fin \in \Scal(V_{\A_\fin})$ be a finite Schwartz function. Let $\varphi=\varphi^V_{\KM} \otimes \varphi_\fin$ and the associated theta lift
\begin{align}
    \Theta^V_\varphi \colon H^{q}(Y,\C) \longrightarrow M_d(\Gamma'_\Q).
\end{align}
We have $\Scal(V_{\A_\fin}) \simeq \Scal(\A_{E,\fin}) \simeq \Scal(\A_{E_1,\fin})\otimes \cdots \otimes \Scal(\A_{E_r,\fin})$ and so we can write $\varphi_\fin$ as a finite sum
\begin{align}
    \varphi_\fin=\sum_{\beta} \varphi_\fin^{\beta,1} \otimes \cdots \otimes \varphi_\fin^{\beta,r}
\end{align}
with $\varphi_\fin^{\beta,i} \in \Scal(\A_{E_i,\fin})$. Similarly, the character $\chi$ on $\Tbf(\Q) \backslash \Tbf(\A)$ can be written as $\chi=\chi_1 \times \cdots \times \chi_r$ where
\begin{align}
    \chi_i \colon \Tbf_i(\Q) \backslash \Tbf_i(\A) \longrightarrow \C^\times
\end{align}
and $\Tbf_i(\Q) \simeq E_i^1$. We set $\varphi^{\beta,i} = \varphi_{\KM}^{E_i} \otimes \varphi_\fin^{\beta,i}$, and as above we get a Hilbert modular form
\begin{align}
    \Theta_{\varphi^{\beta,i}}(\tau_1,\dots, \tau_{d},\chi_i) \in M_{(1,\dots,1)}(\Gamma'_i)
\end{align}
of parallel weight one for a congruence subgroup $\Gamma_i'$ of $\SL_2(F_i)$. We denote by
\begin{align}
    \Theta_{\varphi^{\beta,i}}(\tau^\Delta,\chi_i)=\Theta_{\varphi^{\beta,i}}(\tau, \dots,\tau,\chi_i) \in M_{d_i}(\Gamma'_i \cap \SL_2(\Q))
\end{align} its diagonal restriction, which is a modular form of weight $d_i$.

By the functoriality of the Kudla-Millson we have
\begin{align}
    \restr{\varphi_{\KM}^V}{Y_{E_1} \times \cdots \times Y_{E_r}}=\varphi_{\KM}^{E_1} \wedge \cdots \wedge \varphi_{\KM}^{E_r}.
\end{align}
We can deduce the following from Proposition \ref{seesaw1}.
\begin{prop} \label{splitting}
    We have the seesaw identity
    \begin{align}
    \int_{C_\chi} \Theta^V_{\varphi}(\tau)=\sum_{\beta} \Theta_{\varphi^{\beta,1}}(\tau^\Delta,\chi_1) \cdots \Theta_{\varphi^{\beta,r}}(\tau^\Delta,\chi_r).
    \end{align}
\end{prop}

\subsection{Computations of the Hilbert modular form} \label{specialcases}
Let us compute the Kudla-Millson form
\begin{align}
    \varphi^E_{\KM}=\restr{\varphi^V_{\KM}}{\D^+_T} \in \Omega^{q}(\D^+_T).
\end{align}
Suppose that $E$ is a field. Since $\Tbf$ is of maximal real rank, $S_3$ is empty by Proposition \ref{rankcondi}. Recall that $q=\vert S_1 \vert$ by Proposition \ref{rankcondi}.  At a place $\sigma \in S_2$, the algebra is $E_\sigma \simeq \C$ and the space $\D_\sigma^+$ is a point. So the Kudla-Millson form is of degree $0$ and is simply the Gaussian
\begin{align}
    \varphi_{\KM}^\sigma(v)=e^{-\pi \vert \alpha_\sigma \vert \N_{E_\sigma/F_\sigma}(v)} \in \Omega^{0}(\D_\sigma^{+}).
\end{align}
At a place $\sigma \in S_1$, the algebra is $E_\sigma \simeq \R \times \R$ and $\D_\sigma^+ \simeq \R_{>0}$. An element $v$ in $E_\sigma$ is sent to $(v_\sigma,v'_\sigma)$ in $E_\sigma$. The Kudla-Millson form (see \cite{rbrkmmq}) is given by
\begin{align}
    \varphi^\sigma_{\KM}(v)= \sqrt{\vert \alpha_\sigma \vert} e^{-\pi \vert \alpha_\sigma \vert \left ( (v_\sigma t_\sigma ^{-1})^2+(t_\sigma v_\sigma')^2\right )} \left ( \frac{v_\sigma}{t_\sigma}+t_\sigma v_\sigma' \right ) \frac{dt_\sigma}{t_\sigma} \in \Omega^{1}(\D_\sigma^{+})
\end{align}
for $v \in E_\sigma$.
\subsubsection{}  We define a Schwartz function $\varphi_\8 \in \Scal(E_\R)$ by $\varphi_\8=\prod_\sigma \varphi_\sigma$ where $\varphi_\sigma=\varphi_{\KM}^\sigma$ when $\sigma$ is in $S_2$, and 
 \begin{align}
    \varphi_\sigma(x,x') \coloneqq \sqrt{\vert \alpha_\sigma \vert} e^{- \pi \vert \alpha_\sigma \vert \left ( x^2+(x')^2\right )} \left ( x+x' \right ).
\end{align}
when $\sigma$ is in $S_1$. Then, under the isomorphism $C^\8(\R_{>0}^q) \simeq \Omega^q(\D_T^+)$, we
\begin{align}\label{equality1024}
    (\omega(g_{\8},1)\varphi^E_{\KM})(v)= (\omega(g_{\8},t_\8)\varphi_\8)(v) dt^\times_\8
\end{align}
for $g_\8 \in \SL_2(F_\R)$. Note that since $E$ is a field, any nonzero $v \in E$ must satisfy $v_\sigma v_\sigma'\neq 0$ at any place $\sigma \in S_1$.
Let $\tau \coloneqq (\tau_\sigma)_\sigma \in \HH^{d}$ be the image of $(i, \dots, i)$ by $g_\8$ and write
\begin{align}
 \Tr_{F/\Q}(\tau m) \coloneqq \sum_\sigma \tau_\sigma m_\sigma   
\end{align}
for $m \in F$. { The following Lemma is a classical computation. A similar computation appears in Kudla's work \cite{Kudla1981}, and even Hecke's original paper \cite{hecke}.
}
\begin{lem}
Suppose that $E/F$ is a field extension. We have
    \begin{align*}
        \int_{E_\8^{1}} \omega(g_{\8},t_\8)\varphi_\8(v) \chi_\8(t)dt_\8^\times = 2^q\left ( y_1 \cdots y_d\right )^{\frac{1}{2}} \left (\prod_{\sigma \in S_1} \sgn(v_\sigma+v'_\sigma) \right ) e^{2 i \pi \Tr_{F/\Q}(\tau \alpha m)}
    \end{align*}
    if $m=\N_{E/F}(v)$ is totally positive, and the integral is $0$ otherwise. 
\end{lem}
\begin{proof}
    At a place $\sigma$ in $S_2$ there is no integral, since $\D_\sigma^+$ is a point. At a place $\sigma$ in $S_1$ we have
    \begin{align}
   \int_{E_\sigma^{1}} \omega(1,{t_\sigma})\varphi_\sigma(v_\sigma,v_\sigma')\chi_\sigma(t_\sigma)dt_\sigma^\times & =  \sqrt{\vert \alpha_\sigma \vert} \int_{\R^\times} e^{- \pi \vert \alpha_\sigma \vert \left ( (v_\sigma {t_\sigma}^{-1})^2+({t_\sigma}v_\sigma')^2\right )} \left ( \frac{v_\sigma}{{t_\sigma}}+{t_\sigma}v_\sigma' \right ) \sgn_\sigma(t_\sigma) \frac{d{t_\sigma}}{{t_\sigma}}  \nonumber \\
        & =  2\sqrt{\vert a_\sigma v_\sigma v_\sigma'\vert} \int_0^\8 e^{- \pi \vert a_\sigma v_\sigma v_\sigma' \vert \left ( u^{-2}+u^2\right )} \left ( \frac{\sgn(v_\sigma)}{u}+\sgn(v_\sigma')u \right ) \frac{du}{u} \nonumber 
    \end{align}
    after the substitution $t_\sigma=\sqrt{\left \vert \frac{v_\sigma}{v_\sigma'}\right \vert} u$. Note that the assumption that $\chi_\sigma(t_\sigma)=\sgn(t_\sigma)$ is crucial here, since the integral would be $0$ otherwise. After substituting $\beta=u^2$ we find that this integral is
    \begin{align} \label{integral}
      K_{\frac{1}{2}}(2 \pi \vert \alpha_\sigma v_\sigma v_\sigma'\vert) \sqrt{\vert \alpha_\sigma v_\sigma v_\sigma'\vert} \sgn(v_\sigma+v'_\sigma)
    \end{align}
    where $K_s(w)=\int_0^\8e^{-w(\beta^{-1}+\beta)}\beta^s\frac{d\beta}{\beta}$ is the $K$-Bessel function. We deduce that the integral vanishes when $m_\sigma=v_\sigma v'_\sigma $ is negative at some place {\em i.e.} when $m$ is not totally positive. Using the fact that $K_{\frac{1}{2}}(w)=\sqrt{\frac{2 \pi}{w}}e^{-w}$, equation \eqref{integral} becomes
    \begin{align}
        2\sgn(v_\sigma+v'_\sigma)e^{-2 \pi \vert \alpha_\sigma v_\sigma v_\sigma'\vert}
    \end{align} Thus, taking the product over all the places of $F$ we get
    \begin{align}\label{integral1}
        \int_{E_\8^{1}} \omega(1,t_\8)\varphi_\8(v) \chi_\8(t)dt_\8^\times = 2^q \left (\prod_{\sigma \in S_1} \sgn(v_\sigma+v'_\sigma) \right ) \prod_{\sigma \in S_1 \cup S_2} e^{-2\pi \vert \alpha_\sigma v_\sigma v'_\sigma \vert}.
    \end{align}
\end{proof}

\subsubsection{} Recall that we can write
\begin{align} 
    \Theta^T_{\varphi}(\tau_1, \cdots, \tau_d)= \thetil^T_\varphi(\tau_1, \cdots, \tau_d,t)dt^\times \in \Omega^{q}(Y_T).
\end{align}
By \eqref{equality1024} we have
\begin{align}
\thetil^T_\varphi(\tau_1, \cdots, \tau_d,t)=  (y_1 \cdots y_d)^{-\frac{1}{2}} \sum_{v \in E}  (\omega(g_{\8},t)\varphi)(v)
\end{align}
where $\varphi=\varphi_\8 \otimes \varphi_\fin \in \Scal(\A_E)$ is as above. We get
\begin{align}
    \Theta^T_{\varphi}(\tau_1, \cdots, \tau_d,\chi) & = \frac{2^q}{\vol(K_{T})}\int_{E^1 \backslash E^1_\A} (y_1 \cdots y_d)^{-\frac{1}{2}} \sum_{v \in E}  (\omega(g_{\8},t)\varphi)(v) \chi(t) dt^\times \nonumber \\
    & = \frac{2^q}{\vol(K_{T})} (y_1 \cdots y_d)^{-\frac{1}{2}} \sum_{v \in E^1\backslash E}  \int_{E^1_\A}(\omega(g_{\8},t)\varphi)(v) \chi(t) dt^\times \nonumber \\
    & =  \sum_{m \in F} c^{(m)}_\varphi(C_\chi)e^{2 i \pi \Tr_{F/\Q}(\tau \alpha m)}
\end{align}
where 
\begin{align} \label{fouriercoeffexp}
 c^{(m)}_\varphi(C_\chi)  & \coloneqq \frac{2^q}{\vol(K_{T})} \sum_{\substack{v \in E^1 \backslash E \\ \N_{E/F}(v)=m}} \int_{E^1_\A}(\omega(g_{\8},t)\varphi)(v) \chi(t) dt^\times \nonumber \\
 & =  \frac{2^q}{\vol(K_{T})} \sum_{\substack{v \in E^1 \backslash E \\ \N_{E/F}(v)=m}} \left (\prod_{\sigma \in S_1} \sgn(v_\sigma+v'_\sigma) \right ) \int_{\A_{E,\fin}^1} \varphi_\fin(t^{-1}v)\chi_\fin(t)dt^\times.
\end{align}
Note that $c^{(m)}_\varphi(C_\chi)$ is nonzero only if $m$ is totally positive.
\begin{thm} Let $\chi$ be a character on a anisotropic maximal $\Q$-torus $\Tbf$ of maximal real rank. Then
$\Theta^V_{\varphi}(\tau,C_\chi)$ is a cusp form.
\end{thm}
\begin{proof}
 Suppose that $F$ is a field, and $E$ a quadratic field extension. Then the only $v \in E$ for which $Q_\alpha(v,v)=\N_{E/F}(v,v)=0$ is $v=0$. Since $v_\sigma=v_\sigma'=0$, we see from \eqref{fouriercoeffexp} that the constant term $c^{(0)}_\varphi(C_\chi)$ is zero.
 
If $E/F$ is a product of field extensions $E_i/F_i$, then by Proposition \ref{splitting} the constant term is
\begin{align} \label{constantterm2}
    c^{(0)}_\varphi(C_\chi)=\sum_{\beta} c^{(0)}_{\varphi^{\beta,1}}(C_{\chi_1}) \cdots c^{(0)}_{\varphi^{\beta,r}}(C_{\chi_r}).
    \end{align}
In particular, there must be one of the fields $F_i$ that has a place $\sigma \in S_1$. For this index $i$, the constant term $c^{(0)}_{\varphi^{\beta,i}}(C_{\chi_i})$ is zero for any $\beta$. Hence \eqref{constantterm2} vanishes.
\end{proof}
\subsection{Generating series of intersection numbers} \label{sec:genseries} Let $(V,Q)$ be a quadratic space of even dimension $2d$ over $\Q$. Let $L$ be an even integral lattice and $\varphi_\fin=\id_{\Lhat} \in \Scal(V_{\A_\fin})$ its characteristic function. We describe the geometric features of the Kudla-Millson form, that allows us to have a nice geometric interpretation of Kudla-Millson lift $\Theta_\varphi(\tau,C)$ defined in \eqref{kmlift1}.

\subsubsection{} Any vector $v$ in $V$ with $Q(v,v)>0$ defines a submanifold of codimension $q$
\begin{align}
    \D^+_{v} \coloneqq \left \{ z \in \D^+ \left \vert z \subset v^\perp \right . \right \}.
\end{align}
For every $I$ in the group $\Ccal$ of double cosets 
\begin{align}
    \Ccal = \SO_V(\Q)^+ \backslash \SO_V(\A_\fin)/K_\fin
\end{align}
let $\Gamma_{I,v}$ be the stabilizer of $v$ in $\Gamma_I$. We denote by $C_v(h_I)$ the image of the composition
\begin{align}
\Gamma_{I,v} \backslash \D^+_v \hooklongrightarrow \Gamma_{I,v} \backslash \D^+ \longrightarrow \Gamma_I \backslash \D^+ \eqqcolon Y_I. \end{align}
Note that $C_v(h_I)$ only depends on the orbit of $\Gamma_I v$. For a positive number $n \in \Q$, we define the weighted cycles 
\begin{align}
    C_n(\varphi_\fin,h_I) \coloneqq \sum_{\substack{ v \in \Gamma_I \backslash V \\ Q(v,v)=2n}} \varphi_\fin(h_I^{-1}v)C_v(h_I) \in \Zcal_{pq-q}(Y_I,\partial Y_I,\Z)
\end{align}
and
\begin{align}
    C_n(\varphi) \coloneqq \sum_{I \in \Ccal} C_n(\varphi_\fin,h_I) \in \Zcal_{pq-q}(Y,\partial Y,\Z).
\end{align}
Note that since $L$ is assumed to be integral, the cycle $C_n(\varphi)$ is only nonzero for $n \in \NN_{>0}$.
\subsubsection{}
Let us go back to the Kudla-Millson form introduced in \ref{subseckm} and described its geometric features. Recall that the Kudla-Millson
\begin{align}
    \varphi_{\KM} \in \Omega^{q}(\D^+,\Scal(V_\R))^{\SO_V(\R)^+} \simeq \left [\Omega^{q}(\D^+) \otimes \Scal(V_\R) \right ]^{\SO_V(\R)^+}
\end{align}
is a closed and $\SO_V(\R)^+$-invariant on $\D^+$. In particular, if $\Gamma=\Gamma_I$ (for some $I \in \Ccal$) is a congruence subgroup of $\SO_{V}(\Q)^+$, then the form $\varphi_{\KM}$ is $\Gamma_v$-invariant, where $\Gamma_v$ is in the stabilizer of $v$ in $\Gamma$. Hence, it descends to a form on $\Gamma_v \backslash \D^+$. 

The main geometric feature of the Kudla-Millson form is the following Thom form property. Let $v$ be a positive vector. For any compactly supported form $\omega \in \Omega_c^{pq-q}(\Gamma_v \backslash \D^+)$ of complementary degree we have
\begin{align}
    \int_{\Gamma_v \backslash \D^+} \varphi_{\KM}(v) \wedge \omega= e^{-\pi Q(v,v)} \int_{\Gamma_v \backslash \D^+_v}\omega.
\end{align}
In other words, the form 
\begin{align}
    \varphi^0(v) \coloneqq e^{\pi Q(v,v)}\varphi_{\KM}(v)
\end{align}
is a Poincaré dual to $\Gamma_v \backslash \D^+_v$ in $\Gamma_v \backslash \D^+$.

\subsubsection{} Note that for $g_\tau=\begin{pmatrix}
    \sqrt{y} & \frac{x}{\sqrt{y}} \\ 0 & \frac{1}{\sqrt{y}}
\end{pmatrix}$ we have
\begin{align}
    (\omega(g_\tau)\varphi_{\KM})(\tau)& = y^{\frac{p+q}{2}}\varphi_{\KM}(\sqrt{y}v)e^{i\pi x Q(v,v)}= y^{\frac{p+q}{2}}\varphi^0(\sqrt{y}v)e^{i\pi \tau Q(v,v)}.
\end{align}We can rewrite the theta series \eqref{kmlift2} as
\begin{align}
   \Theta_\varphi(\tau) = y^{-\frac{p+q}{2}} \sum_{v \in V} (\omega(g_\tau)\varphi)(v)  = \sum_{n \in \Q} \Theta_\varphi^{(n)}(\tau)q^n
\end{align}
by grouping the vectors of same length, where
\begin{align}
    \Theta_\varphi^{(n)}(\tau) \coloneqq \sum_{\substack{ v \in V \\ Q(v,v)=2n}} \varphi^{0}(\sqrt{y}v) \varphi_\fin(v)=\sum_{\substack{ v \in L \\ Q(v,v)=2n}} \varphi^{0}(\sqrt{y}v).
\end{align}
The form $\Theta_\varphi^{(n)}(\tau)$ is closed, and for positive $n$ it represents a Poincaré dual of the special cycle $C_n(\varphi)$ in $H^q(Y,\C)$. For negative $n$, the form is exact. When $V$ is anisotropic, the constant term is $\varphi_{0}(0) \varphi_\fin(0)$ where $\varphi_{0}(0)$ represents the Euler class of the tautological bundle over $\D^+$. 

By the work of Kudla and Millson, if $C$ is a class in $H_q(Y,\Z)$ then the period
\begin{align}
    \Theta^V_\varphi(\tau,C) = \int_C \Theta^V_\varphi(\tau)
\end{align}
is a modular form of weight $d=\frac{p+q}{2}$ of level $\Gamma_0(N)$. Since $\Theta_\varphi^{(n)}(\tau)$ is a Poincaré dual to $C_n(\varphi)$, we get that
\begin{align} \label{genseries}
 \int_C \Theta^V_\varphi(\tau) =  c^0_\varphi(C) + \sum_{ n=1}^\8\int_C \Theta_\varphi^{(n)}(\tau) q^n =  c^0_\varphi(C) + \sum_{ n \in \Q_{>0}} \langle C,C_n(\varphi) \rangle q^n.
\end{align} 

When the cycles $C$ and $C_n(\varphi)$ intersect transversely, then $\langle C,C_n(\varphi) \rangle$ is the signed intersection number.

\subsubsection{} Let $E=E_1 \times \cdots \times E_r$ with fixed subalgebra $F=F_1 \times \cdots \times F_r$. So the torus is a product $\Tbf(\Q) = \Tbf_1(\Q) \times \cdots  \times \Tbf_r(\Q)$ where $\Tbf_i(\Q) \simeq E_i^1$. Let $\chi=\chi_1 \times \cdots \times \chi_r$ where $\chi_i \colon \Tbf(\Q) \backslash \Tbf(\A) \longrightarrow \C^\times$. The Schwartz function splits as $\varphi_\fin=\sum_{\beta} \varphi_{\fin}^{\beta,1} \otimes \cdots \otimes \varphi_{\fin}^{\beta,r} \in \Scal(\A_{E_1,\fin}) \otimes \cdots \otimes \Scal(\A_{E_r,\fin})$, as in Proposition \ref{splitting}. 

\begin{thm} \label{mainthmgen}  We have 
\begin{align}
    \sum_{ n =1}^\8 \langle C_\chi,C_n(\varphi) \rangle q^n = \sum_{\beta} \Theta_{\varphi^{\beta,1}}(\tau^\Delta,\chi_1) \cdots \Theta_{\varphi^{\beta,r}}(\tau^\Delta,\chi_r).
    \end{align}
\end{thm}
\begin{proof}
    It follows from combining \eqref{genseries} with Proposition \ref{splitting}.
\end{proof}
\begin{rmk}
    In fact, one can show that that when $F$ is a field, the cycles $C_\chi$ and $C_n(\varphi)$ intersect transversely.
\end{rmk}

\section{Example of biquadratic fields} \label{lowrank}
We consider the setting of a split quadratic space of signature $(2,2)$. The setting is similar to \cite{rbrforum}, but we consider the integral over a product of compact geodesics instead of the the product of a compact geodesic with $\gamma_\8$.

\subsubsection{} Consider the quadratic space $(V,Q)=(\Mat_2(\Q),2\det)$ with the quadratic form \begin{align}
2\det(x)=\Tr(xx^*), \end{align}
where the involution is
\begin{align} \label{involutiont}
    \begin{pmatrix}
        a & b \\ c & d
    \end{pmatrix}^*=\begin{pmatrix}
        d & -b \\ -c & a
    \end{pmatrix}=S^T\begin{pmatrix}
        a & b \\ c & d
    \end{pmatrix}^TS
\end{align}
with $S=\begin{pmatrix}
        0 & 1 \\ -1 & 0
    \end{pmatrix}$ and $A^T$ is the transpose of $A$. In particular, we have $(xy)^*=y^*x^*$. As a bilinear form, the quadratic form is
    \begin{align}
        Q(x,y)=\Tr(xy^\ast)=\Tr(x^\ast y)
    \end{align}
    where $\Tr$ is the matrix trace. The spin group $\GSpin_V(\Q)$ is isomorphic to
\begin{align}
    \GSpin_V(\Q) & \simeq \GL_2(\Q) \times_{\det} \GL_2(\Q)
\end{align}
and consists of matrices $(g_1,g_2)$ with $\det(g_1)=\det(g_2)$. It acts on $\Mat_2(\Q)$ by $\rho(g_1, g_2)y=g_1yg_2^{-1}$ and preserves the quadratic form $\det$. The action of $\GSpin_V(\Q)$ on $V$ induces a short exact sequence
\begin{align}  
     1 \xrightarrow{\ \  \  \ \  } \Q^\times \xrightarrow{\ \  \  \ \  } \GSpin_V(\Q) \xrightarrow{\ \  \rho \ \  } \SO_V(\Q) \xrightarrow{\ \  \ \ \  } 1
\end{align}
where $\Q^\times$ is the center of $\GSpin_V(\Q)$. We have \begin{align}\SO_V(\Q) \simeq \GSpin_V(\Q)/\Q^\times.\end{align}

The connected component $\GSpin_V(\R)^+=\GL_2(\R)^+ \times_{\det} \GL_2(\R)^+$ consists of pairs of matrices with (same) positive determinant. It acts transitively on $\HH \times \HH$ by Möbius transformations in both factors. Let $\Ktil_\8 \subset \GSpin_V(\R)^+$ be the stabilizer of a point in $\HH \times \HH$, so that $\HH \times \HH \simeq \GSpin_V(\R)^+ / \Ktil_\8$.
The stabilizer of $(i,i)$ is $\R_{>0} (\SO(2) \times \SO(2))$. We can extend the Weil representation from the pair $\SO_V(\Q_v) \times \SL_2(\Q_v)$ to the pair $\GSpin_V(\Q_v) \times \SL_2(\Q_v)$ by 
 \begin{align}     \omega(\gtil,h)\varphi(x)=\omega(1,h)\varphi(\rho(\gtil)^{-1}x),
 \end{align}
where $\varphi \in \Scal(\Mat_2(\Q_p))$ is a Schwartz-Bruhat function.

\subsection{Locally symmetric space} For the locally symmetric spaces and the special cycles in this setting, we refer to \cite{rbrforum} for more details. Let $M_0(p) \subset \Mat_2(\Q)$ be the lattice
\begin{align}
    M_0(p) \coloneqq \left \{ \left . \begin{pmatrix}
    a & b \\ c & d
    \end{pmatrix} \in \Mat_2(\Z) \ \right \vert p \mid c, \, (a,p)=1, \, ad-bc>0\right \}
\end{align}
and
 \begin{align}
     \widehat{M}_0(p) \coloneqq \left \{ \left . \begin{pmatrix}
    a & b \\ c & d
    \end{pmatrix} \in \Mat_2(\widehat{\Z})  \ \right \vert a_p \in \Z_p^\times, \, c_p \in p\Z_p \right \},
     \end{align}
that satisfies $\Mat_2(\Q)^+ \cap \widehat{M}_0(p)=M_0(p)$. Let $\varphi_\fin \in \Scal(\Mat_2(\widehat{\Z}))$ be the characteristic function of $\widehat{M}_0(p)$. It is preserved by the open compact $\Ktil_0(p)$, where $\Ktil_0(p) \coloneqq K_0(p) \times_{\det} K_0(p)$ is an open compact in $\GSpin_V(\widehat{\Z})$ and
\begin{align}
K_0(p)=\left \{ \left . \begin{pmatrix}
 a & b \\ c & d
\end{pmatrix} \in \GL_2(\widehat{\Z}) \ \right \vert \  c \in p\widehat{\Z} \right \}. 
\end{align}
We set $\Ktil \coloneqq \Ktil_\8 \Ktil_0(p)$ and we have
\begin{align}
    Y = \GSpin_V(\Q) \backslash \GSpin_V(\A)/\Ktil \nonumber \simeq Y_0(p) \times Y_0(p)
\end{align}
by the strong approximation $\GL_2(\A_\fin)=\GL_2(\Q)^+K_0(p)$.
\subsubsection{} For a vector $M$ in $\Mat_2(\R)$ with positive determinant, the submanifold $\D^+_M$ in $\D^+$ is the image of the map
\begin{align}
    \HH & \hooklongrightarrow \HH \times \HH \nonumber \\
    z & \longmapsto (Mz,z).
\end{align} The special cycles $C_M$ is the image of the immersion
\begin{align}
    \Gamma_M \backslash \HH \longrightarrow Y_0(p) \times Y_0(p)
\end{align}
where $\Gamma_M \coloneqq \Gamma_0(p) \cap M^{-1} \Gamma_0(p) M$. Hence, the special cycles
\begin{align}
    C_n(\varphi)=\sum_{\substack{{M \in \Gamma_0(p) \backslash M_0(p) / \Gamma_0(p)} \\ \det (M)=n}} C_M
\end{align}
are correspondences in $Y_0(p) \times Y_0(p)$.
\subsection{Maximal torus in $\GSpin_V$}
Let $L=\Q(\sqrt{D})$ be a real quadratic field of fundamental discriminant $D$ and suppose that $p$ is split in $L$. Let $\Ocal_L$ be the ring of integers. Let us fix an integer $r \in \Z$ such that $r^2 \equiv D \pmod{4p}$. A form $[a,b,c]=ax^2+bxy+cy^2$ of squarefree discriminant $D$ is called a Heegner form at $p$ if it satisfies $a \equiv 0 \pmod{p}$ and $b \equiv r \pmod{p}$. Let $(u,v)$ be a positive fundamental solution to the Pell equation $u^2-Dv^2=1$. To every primitive Heegner form $[a,b,c]$ at $p$ of discriminant $D=b^2-4ac$ we can associate the matrix
\begin{align}
    \begin{pmatrix}
        u-bv & -2cv \\ 2av & u+bv
    \end{pmatrix} \in \Gamma_0(p).
\end{align}
It is a generator of the (free part of) the orthogonal group of the quadratic form $[a,b,c]$. Its two eigenvalues are $\epsilon_L=u+\sqrt{D}v$ and $\epsilon_L^{-1}=u-\sqrt{D}v$, where $\epsilon_L$ is a fundamental unit in $\Ocal^\times_L$. The eigenvectors are
\begin{align}
    \begin{pmatrix}
        -b + \sqrt{D} \\ 2a
    \end{pmatrix}, \quad \begin{pmatrix}
        -b - \sqrt{D} \\ 2a
    \end{pmatrix}.
\end{align}
The embedding $\phi \colon L \hooklongrightarrow \Mat_2(\Q)$ given by 
\begin{align}
    \phi(\sqrt{D}) = \begin{pmatrix}
        -b & -2c \\ 2a & b
    \end{pmatrix}\in \Gamma_0(p)
\end{align}
is optimal in the sense that $\phi(L^\times) \cap \Gamma_0(p)=\Ocal_L^\times.$

\subsubsection{} Instead of taking a maximal algebraic $\Q$-torus $\Tbf$ in $\SO_V$, let us start with a maximal $\Q$ torus $\Ttil$ in $\GSpin_V$. The image $\Tbf(\Q)=\rho(\Ttil(\Q))$ in $\SO_V(\Q)$ is a maximal $\Q$-torus and we have an exact sequence 
\begin{align}  
     1 \xrightarrow{\ \  \  \ \  } \Q^\times \xrightarrow{\ \  \  \ \  } \Ttil(\Q) \xrightarrow{\ \  \rho \ \  } \Tbf(\Q) \xrightarrow{\ \  \ \ \  } 1.
\end{align}
 Let $L_1$ and $L_2$ be two real quadratic fields with distinct discriminants $D_1$ and $D_2$. Suppose that $p$ is split in both fields. Let $[N_1,r_1,1]$ and $[N_2,r_2,1]$ be the two principal\footnote{It represents the unit in the narrow class group.} Heegner forms where $N_i \coloneqq \frac{r_i^2-D_i}{4}$. Let $\phi_i \colon L_i \hooklongrightarrow \Mat_2(\Q)$ be the two associated optimal embeddings, given by
\begin{align}
    \phi_i(\sqrt{D}) = \begin{pmatrix}
        -r_i & -2 \\ 2N_i & r_i
    \end{pmatrix}\in \Gamma_0(p).
\end{align}
Combining these two embeddings gives an embedding 
\begin{align} \label{Qiso}
    \phi_1 \times \phi_2 \colon L_1^\times \times_{\N} L_2^\times \hooklongrightarrow \GSpin_V(\Q) 
\end{align} and let $\Ttil(\Q) \simeq L_1^\times \times_{\N} L_2^\times$ be the image of this embedding. The product $L_1^\times \times_{\N} L_2^\times$ consists of elements $(\lambda_1,\lambda_2)$ in $L_1 \times L_2$ with the same nonzero norm. We have $\Tbf(\Q) \simeq L_1^\times \times_{\N} L_2^\times/ \Q^\times$.

\subsubsection{} We want to find the \'etale algebra $E$ associated to $\Ttil$. It is the centralizer of $\Ttil$ in $\End(V)=\End(\Mat_2(\Q))$.
\begin{lem}
    The map
\begin{align}
    \Mat_2(\Q) \otimes_\Q \Mat_2(\Q) \longrightarrow \End_\Q(\Mat_2(\Q)) \nonumber \\
    a \otimes b \longmapsto (a \otimes b)x \coloneqq axb^*
\end{align}
is an isomorphism of $\Q$-algebras, where the involution $b^*$ is as in \eqref{involutiont}.
\end{lem}
\begin{proof} The map is a homomorphism of $\Q$-algebras. Since they both have dimension $16$, it is enough to show surjectivity. Let $E_{ij}$ be the standard basis of $\Mat_2(\Q)$, that sends the basis vector $e_a$ in $\Q^2$ to $\delta_{j=a}e_i$. A basis of $\End_{\Q}(\Mat_2(\Q))$ is $\Ecal_{ij}^{kl}$, given by $\Ecal_{ij}^{kl}(E_{ab})=\delta_{a=i}\delta_{b=j}E_{kl}$. We have that $E_{kl}E_{ij}=\delta_{l=i}E_{kj}$. Hence, the element $E_{kl} \otimes E_{ij}^*$ acts by 
    \begin{align}
        (E_{ki} \otimes E_{jl}^*)E_{ab}=E_{ki} E_{ab} E_{jl} = \delta_{a=i}E_{kb} E_{jl}=\delta_{a=i}\delta_{b=j} E_{kl}.
    \end{align}
    Hence $E_{kl} \otimes E_{ij}^\ast$ is sent to $\Ecal_{ij}^{kl}$, and the map is surjective. 
\end{proof}
We can then map $\GSpin_V(\Q)$ in $\End_\Q(\Mat_2(\Q))$ by sending $(g_1,g_2)$ to $\det(g_2)g_1 \otimes g_2$. The map is compatible with the actions of $\GSpin_V(\Q)$ and $\End_\Q(V)$ on $V$.
\begin{prop} \label{involem}
    The \'etale algebra is $E \simeq \Q(\sqrt{D_1},\sqrt{D_2})$ and the involution is $$\epsilon(\sqrt{D_1},\sqrt{D_2})=(-\sqrt{D_1},-\sqrt{D_2}).$$ The fixed subalgebra $F$ is the totally real field $\Q(\sqrt{D_1D_2})$.
\end{prop}
\begin{proof}
Let $\End_\Q(\Mat_2(\Q))$ be the endomorphism ring of $\Mat_2(\Q)$. Let $\epsilon_Q$ be the involution on $\Mat_2(\Q) \otimes \Mat_2(\Q)$ defined by
\begin{align}
    Q((a \otimes b)x,y)=Q(x,\epsilon_Q(a \otimes b)y).
\end{align}
Since $Q(x,y)=\Tr(xy^*)=\Tr(x^\ast y)$, we have $Q(xb,y)=Q(x,yb^*)$and $Q(ax,y)=Q(x,a^*y)$. Hence, the involution is given by $\epsilon_Q(a\otimes b)=a^* \otimes b^*$. The image of $\Ttil(\Q)$ in $\End_\Q(\Mat_2(\Q)) \simeq \Mat_2(\Q) \otimes_\Q \Mat_2(\Q)$ is 
 \begin{align}
     J=\left \{ \left . \N(x_2)x_1 \otimes x_2 \in \Mat_2(\Q) \otimes_\Q \Mat_2(\Q) \ \right \vert (x_1,x_2) \in L_1^{\times} \times_{\N} L_2^{\times} \right \}.
 \end{align} The \'etale algebra of endomorphisms commuting with $J$ is $E=L_1\otimes L_2$. Note that when restricted to $L_i$ the involution $x^*$ acts like the Galois involution $x \mapsto x'$, {\em i.e.} we have $\phi(x')=\phi(x)^*$. Hence, the involution $\epsilon_Q$ restricted to $E$ is $ \epsilon_Q(x_1 \otimes x_2) = x_1' \otimes x_2'$ where $x_i'$ is the Galois conjugate of $x_i$. It follows that the \'etale algebra is $E \simeq \Q(\sqrt{D_1},\sqrt{D_2})$ with the involution sending $\sqrt{D_i}$ to $-\sqrt{D_i}$.
 \end{proof}
 
\subsubsection{} The identity matrix $x_0=\id_2$ is an $L_1 \otimes L_2$ module generator of $\Mat_2(\Q)$. Hence, we have an isomorphism of vector spaces,
 \begin{align} \label{vspaceiso}
     E=L_1 \otimes L_2 & \longrightarrow \Mat_2(\Q) \nonumber \\
     \lambda_1 \otimes \lambda_2 & \longmapsto \phi_1(\lambda_1)\phi_2(\lambda_2)^{\ast}.
 \end{align}
Thus, there exists an $\alpha \in F^\times$ such that $(E,Q_\alpha)\simeq (\Mat_2(\Q),\det)$.
At the level of the torus, we have
\begin{align} \label{quadiso2}
    \Q(\sqrt{D_1},\sqrt{D_2})^1 \simeq \Tbf(\Q) \simeq \rho \left ( L_1^\times \times_{\N} L_2^\times \right ).
\end{align}

\subsubsection{} Over $\R$ the embedding \eqref{Qiso}becomes
\begin{align}
    \phi_1 \times \phi_2 \colon (L_1 \otimes \R)^\times \times_{\N} (L_2 \otimes \R)^\times \hooklongrightarrow \GSpin_V(\R).
\end{align}
The embeddings can be diagonalized over $\R$ as $\phi_i(\lambda)=B_i \begin{pmatrix}
    \lambda_{\sigma} & 0 \\ 0 & \lambda'_{\sigma}
\end{pmatrix}B_i^{-1}$
where $\lambda_{\sigma}=\sigma(\lambda) \in \R$ for an embedding $\sigma$, where $\lambda'$ is the Galois conjugate of $\lambda$ and
\begin{align}
    B_i=\begin{pmatrix}
        -r_i + \sqrt{D_i} & -r_i - \sqrt{D_i} \\ 2N_i & 2N_i
    \end{pmatrix} \in \GL_2(\R)^+.
\end{align}
(If $N_i<0$ then we exchange the two columns to have positive determinant.) Let us fix the point $(z_1,z_2)\coloneqq (B_{1}i,B_{2}i)$. Its stabilizer in $\GSpin_V(\R)^+$ is $\Ktil_\8(z_1,z_2)=(B_{1},B_{2}) \R_{>0}(\SO(2) \times \SO(2)) (B_{1}^{-1},B_{2}^{-1})$. Over $\R$, the torus can be diagonalized as
\begin{align}
    \Ttil(\R)=(B_{1},B_{2}) \Ttil_0(\R) (B_{1}^{-1},B_{2}^{-1})
\end{align} where $\Ttil_0(\R) \simeq (\R^\times)^2 \times_{\N} (\R^\times)^2$ are pairs of diagonal matrices with same determinant. The preimage $\Ktil_{\Tbf,\8} = (\phi_1 \times \phi_2)^{-1}(\Ktil_\8)$ of $\Ktil_\8$ by $\phi_1 \times \phi_2$ (tensored with $\R$) is $\R_{>0}(\pm 1, \pm 1)$. 
Hence 
\begin{align}
    \D_T^+ \simeq \Ttil(\R)^+/\Ktil_{\Tbf,\8} \simeq \R_{>0}^2.
\end{align}At the level of the locally symmetric spaces, the embedding is
\begin{align}\label{localemb}
    (\phi_1 \times \phi_2) \colon (\epsilon_{L_1}^\Z \times \epsilon_{L_2}^\Z) \backslash \R_{>0}^2 & \longrightarrow Y_0(p) \times Y_0(p) \nonumber \\
    (\epsilon_{L_1}^\Z \times \epsilon_{L_2}^\Z)(t_1,t_2) & \longmapsto (\Gamma_0(p) \times \Gamma_0(p))(B_1,B_2)(t_1i,t_2i).
\end{align}
As $t_1$ and $t_2$ range over $\R$, the image of $(t_1i,t_2i)$ is $\gamma_\8 \times \gamma_\8$ and we get the following. 
\begin{prop} The image of \eqref{localemb} is $\gamma_{\Ocal_1} \times \gamma_{\Ocal_2}$, where $\gamma_{\Ocal_i}$ be the geodesic joining $\frac{-r_i-\sqrt{D_i}}{2N_i}$ to $\frac{-r_i+\sqrt{D_i}}{2N_i}$.
\end{prop}
As in \cite{rbrforum}, we use that $\langle C_T,C_n(\varphi)\rangle =\langle \gamma_{\Ocal_1}, T_n \gamma_{\Ocal_2} \rangle $ to deduce that 
\begin{align}
    \sum_{n=1}^\8 \langle \gamma_{\Ocal_1}, T_n \gamma_{\Ocal_2} \rangle q^n
\end{align}
is a diagonal restriction of a Hilbert modular form for a subgroup of $\SL_2(\Q(\sqrt{D_1D_2}))$.

\begin{rmk}\begin{enumerate}
\item The space $(\epsilon_{L_1}^\Z \times \epsilon_{L_2}^\Z) \backslash \R_{>0}^2$ on the left handside of \eqref{localemb} is just one of the connected components of the adelic space $Y_T$. The geodesic $\gamma_{\Ocal_i}$ is the geodesic attached to the identity in the narrow class group $\Ccal_{D_i}^+$. In general, the cycle $C_\chi$ should be a linear combinations of product of geodesics $\gamma_{I} \times \gamma'_{I}$ attached to a class in the class group of the torus $E^1=\Q(\sqrt{D_1},\sqrt{D_2})^1$, and weighted by $\chi$. However, this class group is not the product of the narrow class groups $\Ccal_{D_i}^+$.

\item One could also consider the case where $E$ is a quartic field that is not a biquadratic field. For example if $E$ has two real and one complex place, then the quadratic space is of signature $(3,1)$ and the space $Y$ is a Bianchi modular surface attached to a quadratic field $K$. The image of $E^1_\8$ is a geodesic that should come from a quadratic extension of $K$. It would be interesting to compute explicitely this extension, associated to the initial extension $E/F$.
\item A similar generating series of intersection numbers on a compact { Shimura curve} have been considered by Rickards \cite{rickards}. It should also follow from the seesaw argument that it is the diagonal restriction of a Hilbert modular form.
\end{enumerate}
\end{rmk}

\section{Spans of diagonal restrictions and toric cycles}\label{spans} Let $(V,Q)$ be a rational quadratic space of even dimension and signature $(p,q)$ with $p\geq q >0$. Let $\varphi=\id_{\widehat{L}}$ be the characteristic function of a lattice $\widehat{L} \subset V_{\A_\fin}$, such that $L=\widehat{L} \cap V$ is even and unimodular. 
By Poincaré duality, the pairing
\begin{align} \label{pairing1}
    \langle - , -  \rangle_Y \colon H_q(Y,\C) \times H^q(Y,\C) \longrightarrow \C
\end{align}
is non-degenerate. For a subspace ${ U} \subset H^q(Y,\C)$ let ${ U}^\perp \subset H_q(Y,\C)$ denote the orthogonal complement with respect to the pairing. If we suppose that $q$ is odd, then the Kudla-Millson lift is in fact a lift
\begin{align}
    \Theta=\Theta^V_\varphi \colon H_q(Y,\C) \longrightarrow S_{d}(\SL_2(\Z))
\end{align}
into the space of cusp forms of weight $d=\frac{p+q}{2}$; see \cite[Theorem.~2]{KMIHES}.
On the other hand, the Kudla-Millson lift has an adjoint
\begin{align}
    \overline{\Theta} \colon S_{d}(\SL_2(\Z)) \longrightarrow H^q(Y,\C)
\end{align}
defined by
\begin{align}
    \overline{\Theta}(f)=\int_{\SL_2(\Z) \backslash \HH}\overline{\Theta_\varphi(\tau)} f(\tau)\frac{dxdy}{y^{2-d}}.
\end{align}
It satisfies 
\begin{align}
    \langle \Theta(C),f \rangle_{\operatorname{pet}}=\langle C,\overline{\Theta}(f) \rangle_Y
\end{align}
where the pairing on the left 
\begin{align}
    \langle -,-\rangle_{\operatorname{pet}} \colon S_{d}(\SL_2(\Z)) \times S_{d}(\SL_2(\Z)) \longrightarrow \C
\end{align}
is the Peterson inner product.
\begin{prop} Let $V$ be a quadratic space of dimension $p+q>4$, where $p \geq q>0$ and $q$ is odd. Then the lift $\overline{\Theta}$ is injective and the lift $\Theta$ is surjective.    
\end{prop}
\begin{proof}
    We have $p+q > 4=\max(4,3+r)$ where $r$ is the Witt index. Moreover, we assumed that $L$ is an even unimodular lattice. Hence, the injectivity of $\overline{\Theta}$ is the content of Corollary 1.2 of \cite{bfinj}. Suppose that $f \in S_{d}(\SL_2(\Z))$ is a nonzero cusp form that is not contained in the image of $\Theta$. Without loss of generality we can assume that $f$ is orthogonal to $\im(\Theta)$ with respect to the Peterson inner product, otherwise we replace $f$ by $f-\proj(f)$ where $\proj(f)$ is the orthogonal projection of $f$ onto $\im(\Theta)$ with respect to the Peterson inner product. Hence, for every $C \in H_q(Y,\C)$ we have
    \begin{align}
        0=\langle \Theta(C),f \rangle_{\operatorname{pet}} = \langle C,\overline{\Theta}(f) \rangle_Y.
    \end{align}
    Since the pairing is non-degenerate, this implies that $\overline{\Theta}(f)=0$. By the injectivity of $\overline{\Theta}$ it follows that $f=0$, and that $\Theta$ is surjective. Note that we have $\im(\Theta)=\ker(\overline{\Theta})^\perp$ and $\ker(\Theta)=\im(\overline{\Theta})^\perp$.
\end{proof}

Let $S$ be the set of characters $\chi \colon \Tbf(\Q)\backslash \Tbf(\A) \longrightarrow \C^\times$ with the same assumptions as in the rest of the paper. We define twe following two subspaces. First let  
\begin{align}
    H_T \coloneqq \Span \left \{ \left . C_\chi \ \right \vert  \chi \in S\right \} \subset H_q(Y,\C)
\end{align}
be the homology spanned by the cycles $C_\chi$ for $\chi \in S$. For every cycle $C_\chi$, the lift $\Theta(C_\chi)$ is the diagonal restriction of a product of Hilbert modular forms for $\SL_2(F)$. Let
\begin{align}
    S_T \coloneqq \Span \left \{ \left . \Theta(C_\chi) \ \right \vert  \chi \in S\right \} \subset S_{d}(\SL_2(\Z))
\end{align}
its span.
Let $H_{\cycle} \subset H_{pq-q}(Y,\partial Y,\C)$ be the homology spanned by the special cycles $C_n(\varphi)$. 
Let $H_{\cycle}^\perp \subset H_q(Y,\C)$ be the span of the cycles $C$ that satisfy $\langle C,C_n(\varphi) \rangle=0$ for every $C_n(\varphi)$. 

\begin{cor} Let $V$ be a quadratic space of dimension $p+q>4$, where $p \geq q>0$ and $q$ is odd. We have the following equality
\begin{align}
    \dim \left ( S_{d}(\SL_2(\Z)) \right )-\dim (S_T) = \dim \left ( H_q(Y,\C) \right ) -\dim \left ( \Span \{ H_{\cycle}^\perp , H_T \} \right ).
\end{align}
\end{cor}

\begin{proof} Since $\Theta$ is surjective we have an isomorphism
\begin{align} \label{516}
    \Theta \colon H_q(Y,\C) / \ker (\Theta) \longrightarrow S_{d}(\SL_2(\Z)).
\end{align}
Since $S_T$ is the image of $H_T$, the isomorphism restricts to an isomorphism
\begin{align} \label{517}
    \Theta \colon H_T/\ker(\Theta) \cap H_T \longrightarrow S_T.
\end{align}
We deduce from \eqref{516} that
\begin{align} \label{519}
    \dim \left ( S_{d}(\SL_2(\Z)) \right ) = \dim \left ( H_q(Y,\C) \right ) - \dim (\ker (\Theta))
\end{align}
and from \eqref{517} that
\begin{align} \label{518}
    \dim \left ( S_T \right ) & = \dim \left ( H_T \right ) - \dim (\ker (\Theta) \cap H_T) \nonumber \\
   & =  \dim \left ( \Span \left \{ \ker (\Theta), H_T \right \} \right ) - \dim (\ker (\Theta)).
\end{align} The result follows by taking the difference of \eqref{518} and \eqref{519} and using that $\ker(\Theta)=\im(\overline{\Theta})^\perp=H_{\cycle}^\perp$.
The last equality is due to Kudla and Millson \cite[Theorem.~4.2]{kmcjm}. We recall the proof. First, if $C \in H_{\cycle}^\perp$, then $C \in \ker(\Theta)$ since the Fourier coefficients of $\Theta(C)$ are $\langle C,C_n(\varphi) \rangle=0$. Hence, we have $H_{\cycle}^\perp \subset \ker(\Theta)$.
On the other hand, consider the $n$-th Poincaré series of weight $d$ defined
\begin{align}
 P_n(\tau)=  c\sum_{\gamma \in \Gamma / \Gamma_{\8}} \frac{e^{2i\pi n \gamma \tau}}{j(\gamma,\tau)^{d}}.
\end{align}
The series converges when $p+q>4$ and is a cusp form. The constant $c$ is chosen such that $\langle f,P_n \rangle_{\operatorname{pet}}=a_n(f)$ is the $n$-th Fourier coefficients of $f$. Now suppose that $C$ is in $\im(\overline{\Theta})^\perp$. In particular, for every $n>0$ we have
\begin{align}
    0=\langle C,\overline{\Theta}(P_n) \rangle = \langle \Theta(C),P_n \rangle_{\operatorname{pet}}= \langle C,C_n(\varphi) \rangle.
\end{align}
Hence, we have $\im(\overline{\Theta})^\perp \subset H_{\cycle}^\perp$ and the equality $\ker(\Theta)=\im(\overline{\Theta})^\perp=H_{\cycle}^\perp$ follows from $\ker(\Theta)=\im(\overline{\Theta})^\perp$.
\end{proof}
\printbibliography
\end{document}

%% file: packages.tex
\usepackage{booktabs} 
\usepackage{array} 
\usepackage{paralist} 
\usepackage{verbatim} 
\usepackage{subfig} 
\usepackage{mathtools}
\usepackage{amssymb}
\usepackage{amsthm}
\usepackage{tikz-cd}
\usepackage{mathrsfs}
\usepackage{enumitem}
\usepackage[colorlinks, allcolors=blue]{hyperref}
\usepackage{appendix}
\usepackage{stmaryrd}
\usepackage{titlesec}
\usepackage{bm}
\usepackage{nicefrac}
\usepackage[backend=bibtex,style=alphabetic,sorting=nyvt]{biblatex}
\usepackage[nottoc,notlof,notlot]{tocbibind} 
\usepackage[titles,subfigure]{tocloft} 
\usepackage{graphicx} 
\usepackage{setspace}
\usepackage{minitoc}
\usepackage{mathrsfs}
\usepackage[scr=rsfs,cal=boondox]{mathalfa}

%% file: commands.tex
\newcommand{\8}{\infty}

\newcommand{\Z}{\mathbb{Z}}
\newcommand{\NN}{\mathbb{N}}
\newcommand{\id}{\mathbf{1}}
\newcommand{\idrm}{\mathrm{id}}

\newcommand{\kbf}{\mathcal{k}}

\newcommand{\GSpin}{\operatorname{GSpin}}

\newcommand{\Tbf}{\operatorname{T}}

\newcommand{\Qbar}{\overline{\Q}}

\newcommand{\diag}{\operatorname{diag}}
\newcommand{\cycle}{\operatorname{cycle}}
\newcommand{\Span}{\operatorname{span}}
\newcommand{\SO}{\operatorname{SO}}

\newcommand{\SL}{\operatorname{SL}}
\newcommand{\Sp}{\operatorname{Sp}}

\newcommand{\KM}{\operatorname{KM}}

\newcommand{\GL}{\operatorname{GL}}

\newcommand{\gtil}{\tilde{g}}

\newcommand{\Ktil}{\widetilde{K}}

\newcommand{\Ttil}{\widetilde{\operatorname{T}}}

\newcommand{\D}{\mathcal{D}}
\newcommand{\HH}{\mathscr{H}}
\newcommand{\R}{\mathbb{R}}
\newcommand{\A}{\mathbb{A}}
\newcommand{\Q}{\mathbb{Q}}
\newcommand{\C}{\mathbb{C}}

\newcommand{\Zhat}{\widehat{\Z}}

\newcommand{\Lhat}{\widehat{L}}

\newcommand{\End}{\operatorname{End}}

\newcommand{\vol}{\operatorname{vol}}

\newcommand{\Hom}{\operatorname{Hom}}

\newcommand{\Tr}{\operatorname{Tr}}

\newcommand{\Res}{\operatorname{Res}}

\newcommand{\N}{\operatorname{N}}
\newcommand{\sgn}{\operatorname{sgn}}

\newcommand{\Ocal}{\mathcal{O}}
\newcommand{\Scal}{\mathcal{S}}

\newcommand{\Ucal}{\mathcal{U}}
\newcommand{\Vcal}{\mathcal{V}}
\newcommand{\Wcal}{\mathcal{W}}

\newcommand{\Ccal}{\mathcal{C}}
\newcommand{\Bcal}{\mathcal{B}}

\newcommand{\Zcal}{\mathcal{Z}}
\newcommand{\Qcal}{\mathcal{Q}}

\newcommand{\Ecal}{\mathcal{E}}
\newcommand{\proj}{\textrm{proj}}
\newcommand{\re}{\operatorname{Re}}

\newcommand{\Mat}{\operatorname{Mat}}
\newcommand{\im}{\operatorname{Im}}

\newcommand{\fin}{f}

\newcommand{\hooklongrightarrow}{\lhook\joinrel\longrightarrow}

\newcommand\restr[2]{{
  \left.\kern-\nulldelimiterspace 
  #1 
  \vphantom{\big|} 
  \right|_{#2} 
  }}
 
\newcommand{\thetil}{\widetilde{\Theta}}

\newcommand{\Gal}{\textrm{Gal}}

%% file: theoremlayout.tex
\newtheoremstyle{mytheoremstyle} 
    {1.5em}                    
    {1.em}                    
    {\itshape}                   
    {}                           
    {\normalsize \bfseries}                   
    {.}                          
    {0,5em}                       
    {}  
    
\theoremstyle{mytheoremstyle}

\newtheorem{thm}{Theorem}[section]
\newtheorem*{thm*}{Theorem}
\newtheorem{cor}{Corollary}[thm]
\newtheorem*{cor*}{Corollary}
\newtheorem{lem}[thm]{Lemma}
\newtheorem{prop}[thm]{Proposition}

\newtheorem*{que*}{Question}

\theoremstyle{remark}
\newtheorem{rmk}{Remark}[section]

\numberwithin{equation}{section}